\documentclass[12pt]{article}
\usepackage{latexsym, amssymb, amsthm}
\usepackage{dsfont}

\usepackage{mathrsfs}
\usepackage{amsxtra}
\usepackage{amstext}
\usepackage{amscd}
\usepackage{latexsym}
\usepackage{bbold}

\DeclareMathAlphabet\gothic{U}{euf}{m}{n}

\makeatletter
\def\eqnarray{\stepcounter{equation}\let\@currentlabel=\theequation
\global\@eqnswtrue
\tabskip\@centering\let\\=\@eqncr
$$\halign to \displaywidth\bgroup\hfil\global\@eqcnt\z@
  $\displaystyle\tabskip\z@{##}$&\global\@eqcnt\@ne
  \hfil$\displaystyle{{}##{}}$\hfil
  &\global\@eqcnt\tw@ $\displaystyle{##}$\hfil
  \tabskip\@centering&\llap{##}\tabskip\z@\cr}

\def\endeqnarray{\@@eqncr\egroup
      \global\advance\c@equation\m@ne$$\global\@ignoretrue}

\def\@yeqncr{\@ifnextchar [{\@xeqncr}{\@xeqncr[5pt]}}
\makeatother

\begin{document}
\bibliographystyle{tom}

\newtheorem{lemma}{Lemma}[section]
\newtheorem{thm}[lemma]{Theorem}
\newtheorem{cor}[lemma]{Corollary}
\newtheorem{prop}[lemma]{Proposition}

\theoremstyle{definition}

\newtheorem{ddefinition}[lemma]{Definition}
\newtheorem{remark}[lemma]{Remark}
\newtheorem{exam}[lemma]{Example}

\numberwithin{equation}{section}

\newcommand{\gota}{\gothic{a}}
\newcommand{\gotb}{\gothic{b}}
\newcommand{\gotc}{\gothic{c}}
\newcommand{\gote}{\gothic{e}}
\newcommand{\gotf}{\gothic{f}}
\newcommand{\gotg}{\gothic{g}}
\newcommand{\gothh}{\gothic{h}}
\newcommand{\gotk}{\gothic{k}}
\newcommand{\gotm}{\gothic{m}}
\newcommand{\gotn}{\gothic{n}}
\newcommand{\gotp}{\gothic{p}}
\newcommand{\gotq}{\gothic{q}}
\newcommand{\gotr}{\gothic{r}}
\newcommand{\gots}{\gothic{s}}
\newcommand{\gott}{\gothic{t}}
\newcommand{\gotu}{\gothic{u}}
\newcommand{\gotv}{\gothic{v}}
\newcommand{\gotw}{\gothic{w}}
\newcommand{\gotz}{\gothic{z}}
\newcommand{\gotA}{\gothic{A}}
\newcommand{\gotB}{\gothic{B}}
\newcommand{\gotC}{\gothic{C}}
\newcommand{\gotD}{\gothic{D}}
\newcommand{\gotG}{\gothic{G}}
\newcommand{\gotL}{\gothic{L}}
\newcommand{\gotS}{\gothic{S}}
\newcommand{\gotT}{\gothic{T}}

\newcounter{teller}
\renewcommand{\theteller}{(\alph{teller})}
\newenvironment{tabel}{\begin{list}%
{\rm  (\alph{teller})\hfill}{\usecounter{teller} \leftmargin=1.1cm
\labelwidth=1.1cm \labelsep=0cm \parsep=0cm}
                      }{\end{list}}

\newcounter{tellerr}
\renewcommand{\thetellerr}{(\roman{tellerr})}
\newenvironment{tabeleq}{\begin{list}%
{\rm  (\roman{tellerr})\hfill}{\usecounter{tellerr} \leftmargin=1.1cm
\labelwidth=1.1cm \labelsep=0cm \parsep=0cm}
                         }{\end{list}}

\newcounter{tellerrr}
\renewcommand{\thetellerrr}{(\Roman{tellerrr})}
\newenvironment{tabelR}{\begin{list}%
{\rm  (\Roman{tellerrr})\hfill}{\usecounter{tellerrr} \leftmargin=1.1cm
\labelwidth=1.1cm \labelsep=0cm \parsep=0cm}
                         }{\end{list}}

\newcommand{\vertspace}{\vskip10.0pt plus 4.0pt minus 6.0pt}

\newcounter{proofstep}
\newcommand{\nextstep}{\refstepcounter{proofstep}\vertspace \par
          \noindent{\bf Step \theproofstep} \hspace{5pt}}
\newcommand{\firststep}{\setcounter{proofstep}{0}\nextstep}

\newcommand{\Ni}{\mathds{N}}
\newcommand{\Qi}{\mathds{Q}}
\newcommand{\Ri}{\mathds{R}}
\newcommand{\Ci}{\mathds{C}}
\newcommand{\Ti}{\mathds{T}}
\newcommand{\Zi}{\mathds{Z}}
\newcommand{\Fi}{\mathds{F}}

\renewcommand{\proofname}{{\bf Proof}}

\newcommand{\simh}{{\stackrel{{\rm cap}}{\sim}}}
\newcommand{\ad}{{\mathop{\rm ad}}}
\newcommand{\Ad}{{\mathop{\rm Ad}}}
\newcommand{\alg}{{\mathop{\rm alg}}}
\newcommand{\clalg}{{\mathop{\overline{\rm alg}}}}
\newcommand{\Aut}{\mathop{\rm Aut}}
\newcommand{\arccot}{\mathop{\rm arccot}}
\newcommand{\capp}{{\mathop{\rm cap}}}
\newcommand{\rcapp}{{\mathop{\rm rcap}}}
\newcommand{\diam}{\mathop{\rm diam}}
\newcommand{\divv}{\mathop{\rm div}}
\newcommand{\dom}{\mathop{\rm dom}}
\newcommand{\codim}{\mathop{\rm codim}}
\newcommand{\RRe}{\mathop{\rm Re}}
\newcommand{\IIm}{\mathop{\rm Im}}
\newcommand{\tr}{{\mathop{\rm Tr \,}}}
\newcommand{\Tr}{{\mathop{\rm Tr \,}}}
\newcommand{\Vol}{{\mathop{\rm Vol}}}
\newcommand{\card}{{\mathop{\rm card}}}
\newcommand{\rank}{\mathop{\rm rank}}
\newcommand{\supp}{\mathop{\rm supp}}
\newcommand{\sgn}{\mathop{\rm sgn}}
\newcommand{\essinf}{\mathop{\rm ess\,inf}}
\newcommand{\esssup}{\mathop{\rm ess\,sup}}
\newcommand{\Int}{\mathop{\rm Int}}
\newcommand{\lcm}{\mathop{\rm lcm}}
\newcommand{\loc}{{\rm loc}}
\newcommand{\HS}{{\rm HS}}
\newcommand{\Trn}{{\rm Tr}}
\newcommand{\n}{{\rm N}}
\newcommand{\WOT}{{\rm WOT}}

\newcommand{\at}{@}

\newcommand{\spann}{\mathop{\rm span}}
\newcommand{\one}{\mathds{1}}

\newcommand{\checkQ}{\mbox{{\Large$\check{\mbox{\normalsize$\cQ$}}$}}}

\hyphenation{groups}
\hyphenation{unitary}

\newcommand{\ca}{{\cal A}}
\newcommand{\cb}{{\cal B}}
\newcommand{\cc}{{\cal C}}
\newcommand{\cd}{{\cal D}}
\newcommand{\ce}{{\cal E}}
\newcommand{\cf}{{\cal F}}
\newcommand{\ch}{{\cal H}}
\newcommand{\chs}{{\cal HS}}
\newcommand{\ci}{{\cal I}}
\newcommand{\ck}{{\cal K}}
\newcommand{\cl}{{\cal L}}
\newcommand{\cm}{{\cal M}}
\newcommand{\cn}{{\cal N}}
\newcommand{\co}{{\cal O}}
\newcommand{\cp}{{\cal P}}
\newcommand{\cQ}{\mathcal{Q}}
\newcommand{\cs}{{\cal S}}
\newcommand{\ct}{{\cal T}}
\newcommand{\cx}{{\cal X}}
\newcommand{\cy}{{\cal Y}}
\newcommand{\cz}{{\cal Z}}

\newcommand{\sF}{\mathscr{F}}

\thispagestyle{empty}

\vspace*{1cm}
\begin{center}
{\large\bf Construction of dynamical semigroups \\[10pt]
by a functional regularisation \`{a} la Kato} \\[5mm]
\large  A.F.M. ter Elst and Valentin A. Zagrebnov  \\[25pt]
\normalsize
In memory of Tosio Kato on the 100$^{\rm th}$ anniversary of his birthday

\end{center}


\vspace{5mm}

\begin{list}{}{\leftmargin=1.8cm \rightmargin=1.8cm \listparindent=10mm
   \parsep=0pt}
\item
\small
{\sc Abstract}.
A functional version of the Kato one-parametric regularisation for the construction of
a dynamical semigroup generator of a relative bound one perturbation is introduced.
It does not require that the minus generator of the unperturbed semigroup
is a positivity preserving operator.
The regularisation is illustrated by an example of a boson-number cut-off regularisation.
\end{list}

\let\thefootnote\relax\footnotetext{
\begin{tabular}{@{}l}
\end{tabular}}

\tableofcontents

\section{Introduction} \label{Skg1}

The majority of papers concerning the construction of dynamical semigroups
use the
Kato regularisation method \cite{Kat8}, which goes back to 1954.
This method allows to
treat the case of positive unbounded perturbations with relative bound one and to construct
minimal Markov dynamical semigroups \cite{Dav16}, \cite{Dav17}.

Later a similar tool of the one-parametric regularisation allowed
Chernoff \cite{Chernoff} to prove the following abstract
result in a Banach space $\mathfrak{X}$ for a perturbation of a contraction
$C_0$-semigroup with generator
$A$ and domain $D(A)$.
Let $B$ be a dissipative operator with domain $D(B) \supset D(A)$
and relative bound \textit{one}, that is, there exists a $a \geq 0$ such that
\[
\|B x\|_{\mathfrak{X}} \leq \|A x\|_{\mathfrak{X}} + a \, \|x\|_{\mathfrak{X}}
,  \]
for all $x \in D(A)$.
If the domain $D(B^*)$ of the adjoint operator $B^*$ is dense
in the dual space $\mathfrak{X}'$, then the closure of the operator $A+B$
is the generator of a $C_0$-semigroup.
Note that the hypothesis on $B^*$ is superfluous if the Banach space $\mathfrak{X}$ is reflexive.
See also Okazawa \cite{Okazawa} Theorem~2 for the reflexive case.

The aim of the present paper is to put this tool into an abstract setting that covers the
Kato regularisation method as a particular case.
Our main result is a \textit{functional} version of the Kato regularisation for the construction of
generators when perturbations are with relative bound equal to one.

To produce an application of this result, we construct the generator of a Markov dynamical semigroup
for an open quantum system of bosons \cite{TamuraZagrebnov2} \cite{TamuraZagrebnov3}.
For this system the abstract Kato regularisation
corresponds to the particle-number cut-off in the Fock space.

\medskip

Let $\ch$ be a Hilbert space over $\Ci$.
Consider the Banach space of bounded operators $\cl(\ch)$ and the subspace $\gotC_1 = \gotC_1(\ch) $
of all trace-class operators.
Let $u,v \in \cl(\ch)$.
We say that an operator $u$ is {\bf positive}, in notation $u \geq 0$,
if $(u x,x)_\ch \geq 0$ for all $x \in \ch$.
We write $u \leq v$ if $v-u \geq 0$.
Let $\gotC_1^+ = \{ u \in \gotC_1 : u \geq 0 \} $.
Then $\gotC_1^+$ is a closed cone with trace-norm $\|u\|_{\gotC_1} = \Tr u$ for all $u \in \gotC_1^+$.

Let $\gotC_1^{\rm sa}$ be the Banach space over $\Ri$ of all self-adjoint operators of $\gotC_1$.
An operator $A \colon D(A) \to \gotC_1^{\rm sa}$ with
domain $D(A) \subset \gotC_1^{\rm sa}$ is called {\bf positivity preserving} if
$A u \geq 0$ for all
$u \in D(A)^+$, where $D(A)^+ := D(A) \cap \gotC_1^+$.
A semigroup $(S_t)_{t > 0}$ on $\gotC_1^{\rm sa}$ is called {\bf positivity preserving} if
the map $S_t$ is positivity preserving for all $t > 0$.

Let $D \subset \gotC_1^{\rm sa}$ be a subspace and let
$A, B \colon D \to \gotC_1^{\rm sa}$ be two maps.
Then we write $A \geq 0$ if $A$ is positivity preserving and we write $A \leq B$ if $B-A \geq 0$.
Obviously $\leq$ is a partial ordering on $\cl(\gotC_1^{\rm sa})$.

Let $-H$ be the generator of a positivity preserving contraction $C_0$-semigroup $(e^{-t H})_{t > 0}$
on $\gotC_1^{\rm sa}$.
Let $K \colon D(H) \to \gotC_1^{\rm sa}$ be a positivity preserving
operator and suppose that
\[
\Tr (Ku) \leq \Tr (H u)
\]
for all $u \in D(H)^+$.
We shall prove in Lemma~\ref{lkg231} that this implies that
the operator $K$ is $H$-bounded, but
with relative bound equal to one.
Hence it is an open problem whether operator $-(H - K)$ with $D(H-K)=D(H)$, or a closed extension of 
this operator, is again the generator of a $C_0$-semigroup.

Kato \cite{Kat8} solved this problem for Kolmogorov's evolution equations when the operator $H$
is a positivity preserving map.
To this end he proposed a
regularisation of the perturbation $K$ by replacing it by the one-parametric family
$(r K)_{r\in [0,1)}$
and by taking finally the limit $r \uparrow 1$.

The aim of the present paper is twofold.
First, we wish to consider a more general (functional)
regularisation \`{a} la Kato.
Secondly, we aim to remove the condition that the operator $H$ is
positivity preserving and merely assume the condition that $-H$ is
the generator of a
positivity preserving semigroup.
It is the positivity preserving of the quantum dynamical semigroup
which is indispensable in applications.
We require that the perturbation $K$ of $H$ admits the following type of regularisation.

\begin{ddefinition} \label{dkg102}
Let $(K_\alpha)_{\alpha \in J}$ be a net such that
$K_\alpha \colon D(H) \to \gotC_1^{\rm sa}$
for all $\alpha \in J$.
We call the family $(K_\alpha)_{\alpha \in J}$ a
{\bf functional regularisation} of the operator $K$ if the
following four conditions are valid.
\begin{tabelR}
\item \label{dkg102-1}
$K_\alpha$ is positivity preserving for all $\alpha \in J$.
\item \label{dkg102-2}
For all $\alpha \in J$ there exist $a_\alpha \in [0,\infty)$ and
$b_\alpha \in [0,1)$ such that
\[
\Tr(K_\alpha u) \leq a_\alpha \, \Tr u + b_\alpha \, \Tr(H u)
\]
for all $u \in D(H)^+$.
\item \label{dkg102-3}
$K_\alpha \leq K_\beta \leq K$ for all $\alpha,\beta \in J$ with $\alpha \leq \beta$.
\item \label{dkg102-4}
For all $u \in D(H)^+$ there exists a dense subspace $V$ of $\ch$ such that
$\lim_\alpha ((K_\alpha u)x,x)_\ch = ((K u)x,x)_\ch$
for all $x \in V$.
\end{tabelR}
\end{ddefinition}

As an example one can take $J = [0,1)$ and $K_r = r K$ for all $r \in J$, i.e.\ $a_\alpha = 0$ and
$b_\alpha = r$.
This was used in \cite{Kat8} under the additional assumption that
$H$ is positivity preserving.

The main theorem of this paper is the following.

\begin{thm} \label{tkg101}
Let $-H$ be the generator of a positivity preserving contraction $C_0$-semigroup
on $\gotC_1^{\rm sa}$.
Let $K \colon D(H) \to \gotC_1^{\rm sa}$ be a positivity preserving operator and suppose that
\begin{equation}\label{eqtk102}
\Tr (Ku) \leq \Tr (H u)
\end{equation}
for all $u \in D(H)^+$.
Let $(K_\alpha)_{\alpha \in J}$ be a functional regularisation of $K$.
Set $L_\alpha = H - K_\alpha$ for all $\alpha \in J$.
Then one has the following.
\begin{tabel}
\item \label{tkg101-1}
For all $\alpha \in J$ the operator $- L_\alpha$ is the generator of a positivity preserving
contraction $C_0$-semigroup $(T^\alpha_t)_{t > 0}$ on $\gotC_1^{\rm sa}$.
\item \label{tkg101-2}
If $t > 0$, then $\lim_\alpha T^\alpha_t u$ exists in $\gotC_1^{\rm sa}$ for all
$u \in \gotC_1^{\rm sa}$.
\end{tabel}
For all $t > 0$ define $T_t \colon \gotC_1^{\rm sa} \to \gotC_1^{\rm sa}$ by
$T_t u = \lim_\alpha T^\alpha_t u$.
\begin{tabel}
\setcounter{teller}{2}
\item \label{tkg101-3}
The family $(T_t)_{t > 0}$ is a positivity preserving contraction $C_0$-semigroup on
$\gotC_1^{\rm sa}$
for which the generator is an extension of the operator $-(H - K)$ with domain $D(H)$.
\end{tabel}
\end{thm}

As a corollary we obtain the regularisation theorem of Kato invented in
\cite{Kat8} and which was extended
to dynamical semigroups with unbounded generators by Davies in \cite{Dav17}.

For completeness we recall that the concept of the dynamical semigroups
was motivated by mathematical
studies of the states dynamics of quantum open systems, see \cite{Dav16}.
In a certain approximation it can be
described on an abstract (Banach) space of states by a $C_0$-semigroup of
positive preserving maps.
These semigroups are often called quantum semigroups if in addition
the Kossakowski--Lindblad--Davies Ansatz (see \cite{AttalJoyePillet2}) is satisfied.

In this paper a {\bf dynamical semigroup} is defined to be a positivity preserving
contraction $C_0$-semigroup on
the Banach space $\gotC_1^{\rm sa}$.
The abstract space-states which we consider in this paper consist of self-adjoint
trace-class operators over a complex Hilbert space $\ch$.
In Section~\ref{Skg3} this Hilbert space is the boson Fock space $\sF$.
A semigroup $(T_t)_{t > 0}$ on $\gotC_1^{\rm sa}$ is called {\bf trace preserving}
if $\Tr (T_t u) = \Tr u$ for all $u \in \gotC_1^{\rm sa}$ and $t > 0$.
Then a {\bf Markov dynamical semigroup} is a dynamical semigroup which is
trace preserving.

We prove Theorem~\ref{tkg101} in Section~\ref{Skg2}.
It turns out that the semigroup $(T_t)_{t > 0}$
constructed in Theorem~\ref{tkg101} is minimal in the sense of Kato~\cite{Kat8}.
We conclude Section~\ref{Skg2} with sufficient conditions for $(T_t)_{t > 0}$
being a Markov dynamical semigroup.

In Section~\ref{Skg3} we present an example where the functional regularisation
of the operator $K$ is a particle-number cut-off in the Fock space $\sF$.
We show that the semigroup which is constructed by this regularisation
method is a Markov dynamical semigroup and that it is minimal.
Moreover, the operator $H$ is not positivity preserving.

\section{The regularisation theorem} \label{Skg2}

We start with a lemma concerning bounded positivity preserving operators on
$\gotC_1^{\rm sa}$.

\begin{lemma} \label{lkg200.5}
\mbox{}
\begin{tabel}
\item \label{lkg200.5-1}
Let $u \in \gotC_1^{\rm sa}$.
Then there are unique $v,w \in \gotC_1^+$ such that
$u = v - w$ and $|u| = v + w$, where $|u|$ is the absolute value of $u$.
\item \label{lkg200.5-1.5}
Let $A \in \cl(\gotC_1^{\rm sa})$ be positivity preserving.
Then $\|A u\|_{\gotC_1} \leq \| A |u| \, \|_{\gotC_1}$
for all $u \in \gotC_1^{\rm sa}$.
\item \label{lkg200.5-2}
Let $A,B \in \cl(\gotC_1^{\rm sa})$ be positivity preserving.
Moreover, suppose that $\Tr Au \leq \Tr Bu$ for all $u \in \gotC_1^+$.
Then $\|A\| \leq \|B\|$.
\item \label{lkg200.5-3}
Let $A \in \cl(\gotC_1^{\rm sa})$ be positivity preserving and let $M \geq 0$.
Suppose that
$\Tr(A u) \leq M \, \Tr u$ for all $u \in \gotC_1^+$.
Then $\|A\| \leq M$.
\item \label{lkg200.5-6}
Let $A,B \in \cl(\gotC_1^{\rm sa})$ be positivity preserving and suppose that
$A \leq B$.
Then $A^n \leq B^n$ for all $n \in \Ni$.
\item \label{lkg200.5-5}
Let $(u_\alpha)_{\alpha \in J}$ be a net in $\gotC_1^+$.
Suppose that $u_\alpha \leq u_\beta$ for all $\alpha,\beta \in J$ with
$\alpha \leq \beta$.
Moreover, suppose that $\sup \{ \Tr u_\alpha : \alpha \in J \} < \infty$.
Then the net $(u_\alpha)_{\alpha \in J}$ is convergent in $\gotC_1$.
\end{tabel}
\end{lemma}
\begin{proof}
Statement~\ref{lkg200.5-1} follows from the spectral representation
of the self-adjoint operator $u\in \gotC_1^{\rm sa}$.

\ref{lkg200.5-1.5}.
Let $u \in \gotC_1^{\rm sa}$.
Let $v,w \in \gotC_1^+$ be as in Statement~\ref{lkg200.5-1}.
Then $Av, Aw \in \gotC_1^+$.
So
\[
\|A u\|_{\gotC_1}
= \|A v - A w\|_{\gotC_1}
\leq \|A v\|_{\gotC_1} + \|A w\|_{\gotC_1}
= \Tr A v + \Tr A w
= \Tr A |u|
=  \| A |u| \, \|_{\gotC_1}
.  \]

\ref{lkg200.5-2}
Let $u \in \gotC_1^{\rm sa}$.
Note that $|u| \in \gotC_1^+$ and $\Tr (B-A)|u| \geq 0$ by assumption.
Therefore \ref{lkg200.5-1.5} gives
\[
\|A u\|_{\gotC_1}
\leq \Tr A |u| + \Tr (B-A)|u|
= \Tr B |u|
= \| B |u| \, \|_{\gotC_1}
\leq \|B\| \, \| \, |u| \, \|_{\gotC_1}
= \|B\| \, \|u\|_{\gotC_1}
\]
and the statement follows.

\ref{lkg200.5-3}.
Choose $B = M \, I$ and use Statement~\ref{lkg200.5-2}.

\ref{lkg200.5-6}.
The proof is by induction.
Let $n \in \Ni$ and suppose that $A^n \leq B^n$.
Then
\[
B^{n+1} u
\geq A^n \, B u
= A^n \, (B - A) u + A^{n+1} u
\geq A^{n+1} u
\]
for all $u \in \gotC_1^+$, since $A^n$ is positivity preserving and
$(B-A) u \geq 0$.

\ref{lkg200.5-5}.
Let $M = \sup \{ \Tr u_\alpha : \alpha \in J \} < \infty$.
Let $x \in \ch$. Then
$\alpha \mapsto (u_\alpha x, x)_\ch$ is increasing and bounded above by
$M \, \|x\|_\ch^2$.
So $\lim_\alpha (u_\alpha x, x)_\ch$ exists.
By the polarisation identity $\lim_\alpha (u_\alpha x, y)_\ch$ exists
for all $x,y \in \ch$.
Define the operator $u \colon \ch \to \ch$ such that
\[
(u x, y)_\ch
= \lim_\alpha (u_\alpha x, y)_\ch
\]
for all $x,y \in \ch$.
It is easy to see that $u$ is symmetric and is an element of $\cl(\ch)$.
Clearly
$(u x, x)_\ch = \lim_\alpha (u_\alpha x, x)_\ch \geq 0$ for all $x \in \ch$.
So $u \geq 0$.

Obviously $0 \leq \Tr u_\alpha \leq \Tr u_\beta \leq M$ for all $\alpha,\beta \in J$
with $\alpha \leq \beta$.
So $\lim_\alpha \Tr u_\alpha \leq M$.
Let $N \in \Ni$ and let $ \{ e_n : n \in \{ 1,\ldots,N \} \} $ be an orthonormal set in $\ch$.
Then
\[
\sum_{n=1}^N (u e_n,e_n)_\ch
= \sum_{n=1}^N \lim_\alpha (u_\alpha  e_n,e_n)_\ch
= \lim_\alpha \sum_{n=1}^N (u_\alpha  e_n,e_n)_\ch
\leq \lim_\alpha \Tr u_\alpha
\leq M
.  \]
So $u \in \gotC_1^+$ and
$\Tr u \leq \lim_\alpha \Tr u_\alpha$.
Clearly $\Tr u_\alpha \leq \Tr u$ for all $\alpha \in J$ and
hence $\Tr u = \lim_\alpha \Tr u_\alpha$.
Since $u - u_\alpha \geq 0$ for all $\alpha \in J$, it follows
that $\lim_\alpha \|u - u_\alpha\|_{\gotC_1} = \lim_\alpha \Tr (u - u_\alpha) = 0$.
Therefore $\lim_\alpha u_\alpha = u$ in $\gotC_1$.
\end{proof}

A trace inequality together with positivity preserving gives
$H$-boundedness of a perturbation.

\begin{lemma} \label{lkg231}
Let $-H$ be the generator of a positivity preserving contraction $C_0$-semigroup
on $\gotC_1^{\rm sa}$.
Let $K \colon D(H) \to \gotC_1^{\rm sa}$ be a positivity preserving operator.
Suppose that $\Tr (K u) \leq \Tr (Hu)$ for all $u \in D(H)^+$.
Then $K \, (\lambda \, I + H)^{-1}$ is bounded and
$\|K \, (\lambda \, I + H)^{-1}\| \leq 1$ for all $\lambda > 0$.
Moreover,
$\|K u\|_{\gotC_1} \leq \|H u\|_{\gotC_1}$
for all $u \in D(H)$ and in particular the operator $K$ is
$H$-bounded with relative bound one.
\end{lemma}
\begin{proof}
Let $\lambda > 0$.
Then the resolvent
\[
(\lambda \, I + H)^{-1} = \int_0^\infty e^{-\lambda t} \, S_t \, dt
\]
is a positivity preserving bounded operator on $\gotC_1^{\rm sa}$.
Therefore by composition the operator
$K \, (\lambda \, I + H)^{-1} \colon \gotC_1^{\rm sa} \to \gotC_1^{\rm sa}$
is positivity preserving, hence bounded by \cite{Dav16} Lemma~2.1.
Moreover,
\[
\Tr K \, (\lambda \, I + H)^{-1} u
\leq \Tr H \, (\lambda \, I + H)^{-1} u
= \Tr u - \lambda \Tr (\lambda \, I + H)^{-1} u
\leq \Tr u
\]
for all $u \in \gotC_1^+$.
So $\|K \, (\lambda \, I + H)^{-1}\| \leq 1$ by
Lemma~\ref{lkg200.5}\ref{lkg200.5-3}.
Therefore
$\|K u\|_{\gotC_1}
\leq \|(\lambda \, I + H) u\|_{\gotC_1}
\leq \lambda \, \|u\|_{\gotC_1} + \|H u\|_{\gotC_1}$
for all $u \in D(H)$ and the lemma follows.
\end{proof}

Inequalities between positivity preserving contraction $C_0$-semigroups
are equivalent to inequalities between the resolvents.

\begin{lemma} \label{lkg234}
Let $(S_t)_{t > 0}$ and $(T_t)_{t > 0}$ be two positivity preserving
bounded $C_0$-semigroups with generators $-H$ and $-L$ respectively.
Then the following are equivalent.
\begin{tabeleq}
\item \label{lkg234-1}
$S_t \leq T_t$ for all $t > 0$.
\item \label{lkg234-2}
$(\lambda \, I + H)^{-1} \leq (\lambda \, I + L)^{-1}$ for all
$\lambda > 0$.
\end{tabeleq}
If, in addition, $D(H) \subset D(L)$, then {\rm \ref{lkg234-1}} is
also equivalent to
\begin{tabeleq}
\setcounter{tellerr}{2}
\item \label{lkg234-3}
The operator $H - L$ is positivity preserving.
\end{tabeleq}
\end{lemma}
\begin{proof}
`\ref{lkg234-1}$\Rightarrow$\ref{lkg234-2}'.
This follows from a Laplace transform.

`\ref{lkg234-2}$\Rightarrow$\ref{lkg234-1}'.
It follows from Lemma~\ref{lkg200.5}\ref{lkg200.5-6} that
$(\lambda \, I + H)^{-n} \leq (\lambda \, I + L)^{-n}$
for all $n \in \Ni$.
Let $t > 0$.
Then the Euler formula yields
\[
S_t u
= \lim_{n \to \infty} (I + \tfrac{t}{n} \, H)^{-n} u
\leq \lim_{n \to \infty} (I + \tfrac{t}{n} \, L)^{-n} u
= T_t u
\]
for all $u \in \gotC_1^+$.
So $S_t \leq T_t$.

`\ref{lkg234-1}$\Rightarrow$\ref{lkg234-3}'.
Write $K = H-L$.
Let $u \in D(H)^+$ and $x \in \ch$.
Then
\begin{eqnarray*}
((Ku) x, x)_{\ch}
& = & \lim_{t \downarrow 0}
   \frac{ (((I - S_t) u) x, x)_\ch }{t}
   - \frac{ (((I - T_t) u) x, x)_\ch }{t}  \\
& = & \lim_{t \downarrow 0}  \frac{ (((T_t - S_t) u) x, x)_\ch }{t}
\geq 0
.
\end{eqnarray*}
So $Ku \geq 0$ and $K$ is positivity preserving.

`\ref{lkg234-3}$\Rightarrow$\ref{lkg234-2}'.
Let $\lambda > 0$.
Since the product of positivity preserving maps is positivity preserving,
we obtain that
\[
(\lambda \, I + L)^{-1} - (\lambda \, I + H)^{-1}
= (\lambda \, I + L)^{-1} (H - L) (\lambda \, I + H)^{-1}
\geq 0
.  \]
So $(\lambda \, I + L)^{-1} \geq (\lambda \, I + H)^{-1}$.
\end{proof}

Our first result is a perturbation theorem where the relative bound
is less than one.
We emphasise that we do not assume that the operator $H$ is positivity preserving.

\begin{prop} \label{pkg230}
Let $-H$ be the generator of a positivity preserving contraction $C_0$-semigroup
on $\gotC_1^{\rm sa}$.
Let $K \colon D(H) \to \gotC_1^{\rm sa}$ be a positivity preserving operator.
Suppose there exist $a \in [0,\infty)$ and $b \in [0,1)$ such that
$\Tr (K u) \leq a \, \Tr u + b \, \Tr (Hu)$ for all $u \in D(H)^+$.
Define $L = H - K$.
Then one has the following.
\begin{tabel}
\item \label{pkg230-1}
The operator $L$ is quasi-$m$-accretive.
Moreover, the semigroup generated by $- L$
is a positivity preserving semigroup.
\item \label{pkg230-2}
If in addition
$\Tr (K u) \leq \Tr (Hu)$ for all $u \in D(H)^+$, then $L$ is $m$-accretive.
So $-L$ is the generator of a contraction semigroup.
\item \label{pkg230-3}
If again
$\Tr (K u) \leq \Tr (Hu)$ for all $u \in D(H)^+$, then
\[
\lim_{N \to \infty} \sum_{n=0}^N
   (\lambda \, I + H)^{-1} \Big(  K (\lambda \, I + H)^{-1} \Big)^n u
= (\lambda \, I + L)^{-1} u
\]
for all $u \in \gotC_1^{\rm sa}$ and $\lambda > 0$.
\end{tabel}
\end{prop}
\begin{proof}
First suppose in addition that
\begin{equation}
\Tr (K u) \leq \Tr (Hu)
\label{epkg230;1}
\end{equation}
for all $u \in D(H)^+$.

Let $(S_t)_{t > 0}$ be the semigroup generated by $-H$.
Let $\lambda > 0$.
Then $H (\lambda \, I + H)^{-1} = I - \lambda \, (\lambda \, I + H)^{-1} \leq I$
since $(\lambda \, I + H)^{-1}$ is positivity preserving.
Hence
\begin{eqnarray*}
\Tr(K (\lambda \, I + H)^{-1} u)
& \leq & a \, \Tr((\lambda \, I + H)^{-1} u)
    + b \, \Tr(H (\lambda \, I + H)^{-1} u)  \\
& \leq & a \, \|(\lambda \, I + H)^{-1} u\|_{\gotC_1} + b \, \Tr(H (\lambda \, I + H)^{-1} u)
\leq \Big( \frac{a}{\lambda} + b \Big) \Tr u
\end{eqnarray*}
for all $u \in \gotC_1^+$, where we used that $(\lambda \, I + H)^{-1} u \in D(H)^+$.
Moreover, $K (\lambda \, I + H)^{-1}$ is positivity preserving
as a composition of two positivity preserving maps. Therefore
\[
\|K (\lambda \, I + H)^{-1}\|
\leq \frac{a}{\lambda} + b
\]
by Lemma~\ref{lkg200.5}\ref{lkg200.5-3}.

Let $\lambda \in \Ri$ and suppose that $\lambda > \frac{\textstyle a}{\textstyle 1 - b}$.
Then
$\lambda \, I + L = (I - K (\lambda \, I + H)^{-1}) (\lambda \, I + H)$ is
invertible and
\begin{equation}
(\lambda \, I + L)^{-1}
= \sum_{n=0}^\infty (\lambda \, I + H)^{-1} \Big(  K (\lambda \, I + H)^{-1} \Big)^n
.
\label{epkg230;8}
\end{equation}
If $n \in \Ni_0$, then
$(\lambda \, I + H)^{-1} \Big(  K (\lambda \, I + H)^{-1} \Big)^n \in \cl(\gotC_1^{\rm sa})$
is positivity preserving.
Hence $(\lambda \, I + L)^{-1}$ is positivity preserving.
Moreover, if $u \in \gotC_1^+$ then (\ref{epkg230;8}) yields
$(\lambda \, I + L)^{-1} u \in D(H)^+$.
Now by the addition assumption (\ref{epkg230;1}) one obtains
\begin{eqnarray*}
\Tr u
& = & \Tr (\lambda \, I + L) (\lambda \, I + L)^{-1} u  \\
& = & \lambda \, \Tr (\lambda \, I + L)^{-1} u + \Tr (H - K)
(\lambda \, I + L)^{-1} u  \\
& \geq & \lambda \, \Tr (\lambda \, I + L)^{-1} u
.
\end{eqnarray*}
Therefore
$\Tr ((\lambda \, I + L)^{-1} u) \leq \lambda^{-1} \, \Tr u$.
Since $(\lambda \, I + L)^{-1}$ is positivity preserving, it follows from
Lemma~\ref{lkg200.5}\ref{lkg200.5-3} that $\|(\lambda \, I + L)^{-1}\| \leq \lambda^{-1}$
for all $\lambda > \frac{a}{1 - b}$.
Hence the operator $L$ is $m$-accretive and $- L$ is the generator of a contraction
$C_0$-semigroup.

Let $(T_t)_{t > 0}$ be the semigroup generated by $- L$.
If $t > 0$, then the operator $(I + \frac{t}{n} \, L)^{-1}$ is
positivity preserving for all large $n \in \Ni$.
Hence by the Euler formula one obtains that
$T_t u = \lim_{n \to \infty} (I + \frac{t}{n} \, L)^{-n} u \in \gotC_1^+$
for all $u \in \gotC_1^+$.
Therefore the semigroup $(T_t)_{t>0}$ is positivity preserving.
This proves Statements~\ref{pkg230-1} and \ref{pkg230-2}
of the the proposition if in addition (\ref{epkg230;1}) is valid.
Note that in particular we have proved Statement~\ref{pkg230-2}.

We next prove Statement~\ref{pkg230-1} without the
additional assumption (\ref{epkg230;1}).
We may assume that $b > 0$.
Choose $\omega = \frac{a}{b}$.
Then
\[
\Tr (Ku)
\leq a \, \Tr u + b \, \Tr (Hu)
= b \, \Tr \Big( (\omega \, I + H) u \Big)
\leq \Tr \Big( (\omega \, I + H) u \Big)
\]
for all $u \in D(H)^+$.
So by the above the operator $(\omega \, I + H) - K$ is $m$-accretive and
is the minus generator of a positivity preserving semigroup.
Therefore $L$ is quasi-$m$-accretive and it is the minus generator of a
positivity preserving semigroup.

Finally we prove Statement~\ref{pkg230-3}.
The proof is inspired by the proof of Lemma~7 in \cite{Kat8}.
Fix $\lambda > 0$.
Let $N \in \Ni$ and $r \in (0,1)$.
Then $\|r \, K \, (\lambda \, I + H)^{-1}\| \leq r$ by Lemma~\ref{lkg231}.
So the Neumann series gives
\begin{eqnarray*}
\sum_{n=0}^N (\lambda \, I + H)^{-1} \Big( r \, K (\lambda \, I + H)^{-1} \Big)^n
& \leq & \sum_{n=0}^\infty (\lambda \, I + H)^{-1} \Big( r \, K (\lambda \, I + H)^{-1} \Big)^n  \\
& = & (\lambda \, I + H - r \, K)^{-1}
\leq (\lambda \, I + H - K)^{-1}
,
\end{eqnarray*}
where we use Lemma~\ref{lkg234} in the last step.
Let $u \in \gotC_1^+$.
Taking the limit $r \uparrow 1$ gives
\[
\sum_{n=0}^N (\lambda \, I + H)^{-1} \Big( K (\lambda \, I + H)^{-1} \Big)^n u
\leq (\lambda \, I + H - K)^{-1} u
= (\lambda \, I + L)^{-1} u
.  \]
In particular,
\[
\Tr \Big( \sum_{n=0}^N (\lambda \, I + H)^{-1} \Big( K (\lambda \, I + H)^{-1} \Big)^n u \Big)
\leq \Tr \Big( (\lambda \, I + L)^{-1} u \Big)
.  \]
Then Lemma~\ref{lkg200.5}\ref{lkg200.5-5} gives that
$v = \lim_{N \to \infty} \sum_{n=0}^N (\lambda \, I + H)^{-1} \Big( K (\lambda \, I + H)^{-1} \Big)^n u$
exists in $\gotC_1$.
Then $v \leq (\lambda \, I + L)^{-1} u$.
Conversely, if $N \in \Ni$ and $r \in (0,1)$, then
\[
\sum_{n=0}^N (\lambda \, I + H)^{-1} \Big( r \, K (\lambda \, I + H)^{-1} \Big)^n u
\leq \sum_{n=0}^N (\lambda \, I + H)^{-1} \Big( K (\lambda \, I + H)^{-1} \Big)^n u
\leq v
.  \]
So
\begin{equation}
(\lambda \, I + H - r \, K)^{-1} u
= \sum_{n=0}^\infty (\lambda \, I + H)^{-1} \Big( r \, K (\lambda \, I + H)^{-1} \Big)^n u
\leq v
\label{epkg230;10}
\end{equation}
If $\mu > \frac{a}{1-b}$, then it follows from (\ref{epkg230;8}) that
$\lim_{r \uparrow 1} (\mu \, I + H - r \, K)^{-1} = (\mu \, I + H - K)^{-1}$
in the strong operator topology.
Since $-(H - r \, K)$ is the generator of a contraction semigroup for all $r \in (0,1]$,
it follows from \cite{Dav1} Theorem~3.17 that
$\lim_{r \uparrow 1} (\mu \, I + H - r \, K)^{-1} = (\mu \, I + H - K)^{-1}$
in the strong operator topology for all $\mu > 0$.
Then taking the limit $r \uparrow 1$ in (\ref{epkg230;10})
gives $(\lambda \, I + H - K)^{-1} u \leq v$.
So $v = (\lambda \, I + H - K)^{-1} u$ and the proof is complete.
\end{proof}

We are now able to prove Theorem~\ref{tkg101} regarding the
functional regularisation of the perturbation of~$H$ and we shall
prove that the perturbed semigroup is a dynamical semigroup.

\begin{thm} \label{tkg203}
Let $-H$ be the generator of a positivity preserving contraction $C_0$-semigroup
on $\gotC_1^{\rm sa}$.
Let $K \colon D(H) \to \gotC_1^{\rm sa}$ be a positivity preserving operator
and suppose that
\[
\Tr (Ku) \leq \Tr (H u)
\]
for all $u \in D(H)^+$.
Let $(K_\alpha)_{\alpha \in J}$ be a functional regularisation of $K$.
Set $L_\alpha = H - K_\alpha$ for all $\alpha \in J$.
Then one has the following.
\begin{tabel}
\item \label{tkg203-0.5}
If $\alpha \in J$, then the operator $L_\alpha$ is $m$-accretive and
the semigroup $(T^\alpha_t)_{t>0}$ generated by $- L_\alpha$ is a positivity preserving
contraction semigroup.
\item \label{tkg203-1}
If $\alpha,\beta \in J$ and $\alpha \leq \beta$, then
$T^\alpha_t \leq T^\beta_t$ for all $t > 0$.
\item \label{tkg203-2}
If $t > 0$, then $\lim_\alpha T^\alpha_t u$ exists in $\gotC_1^{\rm sa}$
for all $u \in \gotC_1^{\rm sa}$.
\end{tabel}
For all $t > 0$ define $T_t \colon \gotC_1^{\rm sa} \to \gotC_1^{\rm sa}$ by
$T_t u = \lim_\alpha T^\alpha_t u$.
\begin{tabel}
\setcounter{teller}{3}
\item \label{tkg203-3}
If $t > 0$, then the map $T_t$ is positivity preserving.
\item \label{tkg203-4}
$(T_t)_{t > 0}$ is a contraction $C_0$-semigroup on $\gotC_1^{\rm sa}$.
\end{tabel}
Now let $-L$ be the generator of the $C_0$-semigroup $(T_t)_{t > 0}$.
\begin{tabel}
\setcounter{teller}{5}
\item \label{tkg203-5}
Let $\lambda > 0$.
Then $\lim_\alpha (\lambda \, I + L_\alpha)^{-1} u = (\lambda \, I + L)^{-1} u$
in $\gotC_1$ for all $u \in \gotC_1^{\rm sa}$.
\item \label{tkg203-6}
The operator $L$
is an extension of operator $H - K$.
\end{tabel}
\end{thm}
\begin{proof}
\ref{tkg203-0.5}.
Condition (\ref{eqtk102}) and Definition~\ref{dkg102}\ref{dkg102-3} imply that
\[
\Tr (K_\alpha u) \leq \Tr (H u)
\]
for all $u \in D(H)^+$.
Using Definition~\ref{dkg102}\ref{dkg102-1} and \ref{dkg102-2}, we may apply
Proposition~\ref{pkg230} to $H$ and $K_\alpha$
in order to obtain the statement.

\ref{tkg203-1}.
Let $\alpha,\beta \in J$ with $\alpha \leq \beta$.
Then
$L_\alpha - L_\beta = K_\beta - K_\alpha \geq 0$.
Moreover, $D(L_\alpha) = D(L_\beta)$.
Now the statement follows from Lemma~\ref{lkg234}\ref{lkg234-3}$\Rightarrow$\ref{lkg234-1}.

\ref{tkg203-2}.
Fix $t > 0$. Let $u \in \gotC_1^+$.
Then \ref{tkg203-1} yields $0 \leq T^\alpha_t u \leq T^\beta_t u$ for all $\alpha,\beta \in J$
with $\alpha \leq \beta$.
Moreover, $\Tr(T^\alpha_t u) = \|T^\alpha_t u\|_{\gotC_1} \leq \|u\|_{\gotC_1}$
for all $\alpha \in J$, since $T^\alpha_t$ is a contraction
by Statement~\ref{tkg203-0.5}.
So $\lim_\alpha T^\alpha_t u$ exists in $\gotC_1^{+}$ by Lemma~\ref{lkg200.5}\ref{lkg200.5-5}.
Then the statement for all $u \in \gotC_1^{\rm sa}$ follows from Lemma~\ref{lkg200.5}\ref{lkg200.5-1}.

\ref{tkg203-3}.
Since the semigroup $(T^\alpha_t)_{t > 0}$ is positivity preserving
for all $\alpha \in J$ by Proposition~\ref{pkg230},
the assertion follows from \ref{tkg203-2} and from the limit
$T_t u = \lim_\alpha T^\alpha_t u$ for all $u \in \gotC_1^+$.

\ref{tkg203-4}.
Let $t > 0$.
Then $\Tr T_t u = \lim_\alpha \Tr(T^\alpha_t u) = \lim_\alpha\|T^\alpha_t u\|_{\gotC_1}
\leq \|u\|_{\gotC_1} = \Tr u$ for all $u \in \gotC_1^+$.
Since $T_t$ is positivity preserving by Statement~\ref{tkg203-3}, it follows from
Lemma~\ref{lkg200.5}\ref{lkg200.5-3} that $T_t$ is a contraction.
Next, taking the limit \ref{tkg203-2} one verifies the semigroup property of the family
$(T_t)_{t > 0}$.

To check the strong continuity of the semigroup $(T_t)_{t > 0}$,
let $u \in \gotC_1^+$, $t > 0$ and $\alpha \in J$.
Then $S_t \leq T^\alpha_t$ by Lemma~\ref{lkg234}\ref{lkg234-3}$\Rightarrow$\ref{lkg234-1}
and Definition~\ref{dkg102}\ref{dkg102-1}.
So $T^\alpha_t - S_t \geq 0$.
Since $T^\alpha_t$ is a contraction, it follows that
\[
\|T^\alpha_t u - S_t u\|_{\gotC_1}
= \Tr ((T^\alpha_t - S_t) u)
\leq \Tr u - \Tr S_t u
= \Tr ((I - S_t) u)
.  \]
Taking the limit over $\alpha$ one gets $\|T_t u - S_t u\|_{\gotC_1} \leq \Tr ((I - S_t) u)$.
Since $(S_t)_{t>0}$ is a strongly continuous semigroup on $\gotC_1^{\rm sa}$
and $\Tr$ is continuous from $\gotC_1^{\rm sa}$ into $\Ri$, one deduces that
$\lim_{t \downarrow 0} \|T_t u - S_t u\|_{\gotC_1} = 0$.
But $\lim_{t \downarrow 0} S_t u = u$ in $\gotC_1^{\rm sa}$.
So $\lim_{t \downarrow 0} T_t u = u$ in $\gotC_1^{\rm sa}$.
The extension of the last limit to all $u \in \gotC_1^{\rm sa}$
follows from Lemma~\ref{lkg200.5}\ref{lkg200.5-1}.

\ref{tkg203-5}.
Let $u \in \gotC_1^+$.
Let $\alpha,\beta \in J$ with $\alpha \leq \beta$.
Then \ref{tkg203-1} and the definition of $T$ give
$0 \leq T^\alpha_t u \leq T^\beta_t u \leq T_t u$ for all $t > 0$.
Hence
\begin{equation}
0 \leq (\lambda \, I + L_\alpha)^{-1} u \leq (\lambda \, I + L_\beta)^{-1} u
\leq (\lambda \, I + L)^{-1} u
.
\label{etkg203;2}
\end{equation}
Therefore by Lemma~\ref{lkg200.5}\ref{lkg200.5-5} it
follows that $\lim_\alpha (\lambda \, I + L_\alpha)^{-1} u$ exists in $\gotC_1$.
We next show that the limit is equal to $(\lambda \, I + L)^{-1} u$.

Let $x \in \ch$ and $N \in (1,\infty)$.
For all $\alpha \in J$ define $f_\alpha,f \colon [0,N] \to [0,\infty)$ by
\[
f_\alpha(t) = e^{-\lambda t} \, ((T^\alpha_t u)x, x)_\ch
\quad \mbox{and} \quad
f(t) = e^{-\lambda t} \, ((T_t u)x, x)_\ch
.  \]
Then $f_\alpha,f$ are continuous.
Moreover, $(f_\alpha)_{\alpha \in J}$ is increasing and
$\lim_\alpha f_\alpha = f$ pointwise.
Since $[0,N]$ is compact it follows that $\lim f_\alpha = f$ uniformly.
Therefore
\[
\lim_\alpha \int_0^N e^{-\lambda t} \, ((T^\alpha_t u)x, x)_\ch \, dt
= \int_0^N e^{-\lambda t} \, ((T_t u)x, x)_\ch \, dt
.  \]
This is for all $N \in (1,\infty)$.
If $\alpha \in J$, then the semigroup $(T^\alpha_t)_{t>0}$ is
a contraction by Statement~\ref{tkg203-0.5}.
Hence $\|T^\alpha_t u\|_{\cl(\ch)} \leq \|T^\alpha_t u\|_{\gotC_1} \leq \|u\|_{\gotC_1}$
for all $t > 0$.
Similarly $\|T_t u\|_{\cl(\ch)} \leq \|u\|_{\gotC_1}$ for all $t > 0$.
Hence
\begin{eqnarray*}
\lim_\alpha (((\lambda \, I + L_\alpha)^{-1} u)x, x)_\ch
& = & \lim_\alpha \int_0^\infty e^{-\lambda t} \, ((T^\alpha_t u)x, x)_\ch \, dt  \\
& = & \int_0^\infty e^{-\lambda t} \, ((T_t u)x, x)_\ch \, dt
= (((\lambda \, I + L)^{-1} u)x, x)_\ch
.
\end{eqnarray*}
This is for all $x \in \ch$.
Polarisation gives
\[
\lim_\alpha (((\lambda \, I + L_\alpha)^{-1} u)x, y)_\ch
= (((\lambda \, I + L)^{-1} u)x, y)_\ch
\]
for all $x,y \in \ch$.
Since we know that
$\lim_\alpha (\lambda \, I + L_\alpha)^{-1} u$ exists in $\gotC_1$,
we conclude that
\[
\lim_\alpha (\lambda \, I + L_\alpha)^{-1} u
= (\lambda \, I + L)^{-1} u
\]
in $\gotC_1$.

Finally Lemma~\ref{lkg200.5}\ref{lkg200.5-1} implies that
$\lim_\alpha (\lambda \, I + L_\alpha)^{-1} u = (\lambda \, I + L)^{-1} u$ for all
$u \in \gotC_1^{\rm sa}$ and the proof of Statement~\ref{tkg203-5} is complete.

Before we can prove Statement~\ref{tkg203-6}, we need two lemmata.
In the next lemma we use for the first time the convergence in
Definition~\ref{dkg102}\ref{dkg102-4}.

\begin{lemma} \label{lkg245}
Let $\lambda > 0$ and $u \in \gotC_1^{\rm sa}$.
Then
$\lim_\alpha K_\alpha \, (\lambda \, I + H)^{-1} u = K \, (\lambda \, I + H)^{-1} u$
in $\gotC_1$.
\end{lemma}
\begin{proof}
By Lemma~\ref{lkg200.5}\ref{lkg200.5-1} we may assume that $u \in \gotC_1^+$.
Then the net $( K_\alpha \, (\lambda \, I + H)^{-1} u )_{\alpha \in J}$
is increasing and
$\Tr K_\alpha \, (\lambda \, I + H)^{-1} u \leq \Tr K \, (\lambda \, I + H)^{-1} u
\leq \|u\|_{\gotC_1}$ by Definition~\ref{dkg102}\ref{dkg102-3} and Lemma~\ref{lkg231}.
So $v = \lim_\alpha K_\alpha \, (\lambda \, I + H)^{-1} u$ exists in $\gotC_1$ by
Lemma~\ref{lkg200.5}\ref{lkg200.5-5}.
By Definition~\ref{dkg102}\ref{dkg102-4} there exists a
dense subspace $V$ of $\ch$ such that
$\lim_\alpha ((K_\alpha u)x,x)_\ch = ((K u)x,x)_\ch$
for all $x \in V$.
If $x \in V$, then
\[
( v x,x )_{\ch}
= \lim_\alpha ( (K_\alpha \, (\lambda \, I + H)^{-1} u) x,x )_{\ch}
= ( (K \, (\lambda \, I + H)^{-1} u) x,x )_{\ch}
.  \]
So by polarisation one deduces that
$v x
= (K \, (\lambda \, I + H)^{-1} u) x$ in $\ch$ for all $x \in V$.
Hence $v = K \, (\lambda \, I + H)^{-1} u$ by continuity and
$\lim_\alpha K_\alpha \, (\lambda \, I + H)^{-1} u = (K \, (\lambda \, I + H)^{-1} u$
in $\gotC_1$.
\end{proof}

Proposition~\ref{pkg230}\ref{pkg230-3} is applicable to the
operators $K_\alpha$.
We next show a version for the full perturbation~$K$.

\begin{lemma} \label{lkg246}
Let $\lambda > 0$ and $u \in \gotC_1^{\rm sa}$.
Then
\[
\lim_{N \to \infty}
   \sum_{n=0}^N  (\lambda \, I + H)^{-1} \Big(  K \, (\lambda \, I + H)^{-1} \Big)^n u
= (\lambda \, I + L)^{-1} u
\]
in $\gotC_1$.
\end{lemma}
\begin{proof}
For all $\lambda > 0$, $N \in \Ni$ and $\alpha \in J$ define
\begin{eqnarray*}
R_{\alpha,N}(\lambda)
& = & \sum_{n=0}^N  (\lambda \, I + H)^{-1} \Big(  K_\alpha \, (\lambda \, I + H)^{-1} \Big)^n
\quad \mbox{and}   \\
R_N(\lambda)
& = & \sum_{n=0}^N  (\lambda \, I + H)^{-1} \Big(  K \, (\lambda \, I + H)^{-1} \Big)^n
.
\end{eqnarray*}
Again it suffices to consider $u \in \gotC_1^+$.
Let $N \in \Ni$.
Then
\[
R_{\alpha,N}(\lambda) u
\leq \sum_{n=0}^\infty
   (\lambda \, I + H)^{-1} \Big(  K_\alpha \, (\lambda \, I + H)^{-1} \Big)^n u
= (\lambda \, I + L_\alpha)^{-1} u
\leq (\lambda \, I + L)^{-1} u
\]
for all $\alpha \in J$ by Proposition~\ref{pkg230}\ref{pkg230-3} and
(\ref{etkg203;2}).
Note that $K_\alpha \, (\lambda \, I + H)^{-1}$
is a contraction for all $\alpha \in J$.
Take the limit over $\alpha$.
Then Lemma~\ref{lkg245} gives
$R_N(\lambda) u \leq (\lambda \, I + L)^{-1} u$.
Therefore $\Tr R_N(\lambda) u \leq \Tr (\lambda \, I + L)^{-1} u$ for all $N \in \Ni$.
It follows from Lemma~\ref{lkg200.5}\ref{lkg200.5-5} that
$v = \lim_{N \to \infty} R_N(\lambda) u$ exists in $\gotC_1$.
Then $v \leq (\lambda \, I + L)^{-1} u$.
If $N \in \Ni$ and $\alpha \in J$, then
Definition~\ref{dkg102}\ref{dkg102-3} and
Lemma~\ref{lkg200.5}\ref{lkg200.5-6} give
$R_{\alpha,N}(\lambda) u \leq R_N(\lambda) u \leq v$.
So $(\lambda \, I + L_\alpha)^{-1} u \leq v$ by
Proposition~\ref{pkg230}\ref{pkg230-3} and
consequently $(\lambda \, I + L)^{-1} u \leq v$ by Theorem~\ref{tkg203}\ref{tkg203-5}.
So $v = (\lambda \, I + L)^{-1} u$ as required.
\end{proof}

Now we are able to prove Statement~\ref{tkg203-6} of Theorem~\ref{tkg203}
as in \cite{Kat8} Lemma~8.

\medskip

\noindent
{\bf Proof of Theorem~\ref{tkg203}\ref{tkg203-6}}.
For all $N \in \Ni$ write
\[
R_N = \sum_{n=0}^N  (I + H)^{-1} \Big(  K \, (I + H)^{-1} \Big)^n
.  \]
Then $R_N = (I + H)^{-1} + R_{N-1} \, K \, (I + H)^{-1}$
for all $N \in \Ni$ with $N \geq 2$.
Let $u \in D(H)$.
Then $R_N \, (I + H) u = u + R_{N-1} \, K u$.
Taking the limit $N \to \infty$ and using Lemma~\ref{lkg246} gives
$(I + L)^{-1} \, (I + H) u = u + (I + L)^{-1} \, K u$.
Hence $(I + L)^{-1} (I + H - K) u = u$ and $u \in D(L)$.
Moreover, $(I + H - K) u = (I + L) u$.
So $L$ is an extension of $H - K$.
The proof of Theorem~\ref{tkg203} is complete.
\end{proof}

Let $L$ and the semigroup $(T_t)_{t > 0}$ be as in Theorem~\ref{tkg203}.
Then $(T_t)_{t > 0}$ is a dynamical semigroup.
It satisfies the following minimality.

\begin{thm} \label{tkg204}
Let $L'$ be an extension of the operator $(H - K)$, $D(H - K) = D(H)$,
such that $-L'$ generates a positivity preserving
$C_0$-semigroup $(T'_t)_{t > 0}$.
Then $T'_t \geq T_t$ for all $t > 0$.
\end{thm}
\begin{proof}
By adding a large constant to the operator $H$, we may assume that
$(T'_t)_{t > 0}$ is a bounded semigroup.
Let $\lambda > 0$ and $\alpha \in J$.
Note that the range
${\rm{Ran }}((\lambda \, I + L_{\alpha})^{-1}) = D(H) \subset D(L_{\alpha})\cap D(L')$.
Hence by the resolvent identity we have
\begin{eqnarray*}
(\lambda \, I + L')^{-1} - (\lambda \, I + L_\alpha)^{-1}
& = & (\lambda \, I + L')^{-1}(L_\alpha - L') (\lambda \, I + L_\alpha)^{-1}  \\
& = & (\lambda \, I + L')^{-1}(K - K_\alpha) (\lambda \, I + L_\alpha)^{-1}  \geq 0 ,
\end{eqnarray*}
since the resolvents and the operator $K - K_\alpha$ are positivity preserving.
Using Theorem~\ref{tkg203}\ref{tkg203-5} one gets
$(\lambda \, I + L')^{-1} \geq (\lambda \, I + L)^{-1}$.
Then the theorem is a consequence of
Lemma~\ref{lkg234} ~\ref{lkg234-2}$\Rightarrow$\ref{lkg234-1}.
\end{proof}

Theorem~\ref{tkg204} states similarly to \cite{Kat8} Lemma~9 that the semigroup
$(T_t)_{t > 0}$ constructed in Theorem \ref{tkg203} by
the functional regularisation $(K_\alpha)_{\alpha \in J}$ is \textit{minimal}.

\begin{cor} \label{ckg206}
Let $-H$ be the generator of a positivity preserving contraction $C_0$-semigroup
on $\gotC_1^{\rm sa}$.
Let $K \colon D(H) \to \gotC_1^{\rm sa}$ be a positivity preserving operator
and suppose that
\[
\Tr (Ku) \leq \Tr (H u)
\]
for all $u \in D(H)^+$.
Let $(K_\alpha)_{\alpha \in J}$ and $(K'_\alpha)_{\alpha \in J'}$
be two functional regularisations of $K$.
Let $(T_t)_{t > 0}$ and $(T'_t)_{t > 0}$ be the semigroups as in
Theorem~\ref{tkg203} using $(K_\alpha)_{\alpha \in J}$ and $(K'_\alpha)_{\alpha \in J'}$,
respectively.
Then $T_t = T'_t$ for all $t > 0$.
\end{cor}

Thus the constructed semigroup is independent of the functional regularisation.

\medskip

We conclude this section by a condition which ensures that the
semigroup $(T_t)_{t > 0}$ constructed in Theorem~\ref{tkg203} is also
trace-preserving and hence is a Markov dynamical semigroup.

\begin{thm} \label{tkg205}
Adopt the notation and assumptions as in Theorem~\ref{tkg101}.
Suppose that
\begin{equation}\label{eqtk203;1}
\Tr(Hu - Ku) = 0
\end{equation}
for all $u \in D(H)$ and that $D(H)$ is a core for the generator $- L$, which is defined by 
Theorem \ref{tkg101}. Then the semigroup $(T_t)_{t > 0}$ is trace preserving.
\end{thm}
\begin{proof}
The proof is a variation of the proof of \cite{Dav17} Theorem~3.2.
Condition~(\ref{eqtk203;1}) states that $\Tr L u = 0$ for all $u \in D(H)$.
Because $D(H)$ is a core for $L$ one deduces that
$\Tr L u = 0$ for all $u \in D(L)$.
Let $u \in D(L)$.
Since the semigroup $(T_t)_{t > 0}$ maps $D(L)$ into $D(L)$,
one also gets $\Tr L \, T_t u = 0$ for all $t > 0$.
Then differentiability of the function $t \mapsto T_t u$ from
$(0,\infty)$ into $\gotC_1$ yields
$\partial_t \, \Tr T_t u  =  - \Tr L \, T_t u  = 0 $
for all $t > 0$.
Hence $\Tr T_t u = \Tr u$ for all $t > 0$.
Since $D(L)$ is dense in $\gotC_1^{\rm sa}$,
the latter also holds for all $u \in \gotC_1^{\rm sa}$.
\end{proof}

\section{Example}  \label{Skg3}

In this section we consider an example of a functional regularisation by boson-number
cut-off in a Fock space~$\sF$.
We construct in this way a dynamical semigroup which is minimal and Markovian.
The unperturbed positivity preserving
$C_0$-semigroup in the example has a (minus) generator which fails to be positivity
preserving.

\subsection{Open boson system}\label{Sskg3}

This example is motivated by the model of an open boson system studied in
\cite{TamuraZagrebnov2} and \cite{TamuraZagrebnov3}.

Let  $b$ and $b^*$ be the boson annihilation and the creation operators defined in the
Fock space $\sF$ generated by a cyclic vector $\Omega$.
That is, the Hilbert space $\sF$ has an orthonormal basis $(e_n)_{n \in \Ni_0}$
with $e_0 = \Omega$
and the Bose operators $b, b^*$ are defined by
\[
b \, e_n = \sqrt{n} \, e_{n-1}
\quad \mbox{and} \quad b^* \, e_n = \sqrt{n+1} \, e_{n+1}
  \]
for all $n \in \Ni_0$, with domain
$D(b) = D(b^*) = \{\psi \in \sF : \sum_{n=0}^\infty  n \, |(\psi, e_n )_{\sF}|^2 < \infty \} $,
where we set $e_{-1} = 0$.
The Bose operators satisfy the commutation relation
$(b \, b^* - b^* \, b) \psi = \psi$ for all $\psi \in D(b^* \, b)$.

The \textit{isolated} system that we consider is a one-mode quantum oscillator with equidistant
discrete spectrum with
spacing $E >0$ defined by
\begin{equation} \label{h-E}
h = E \, b^* \, b
\end{equation}
and domain
\[
D(h)  =
\{\psi \in \sF : \sum_{n=0}^\infty  n^2 \, |(\psi, e_n )_{\sF}|^2 < \infty \}
.  \]
The {\bf number operator}
\[
\hat n := b^* \, b ,
\]
with $D(\hat n ) =D(h) \subset \mathcal{F}$,
counts the number of bosons $(\hat n  \, \psi, \psi)_{\sF}$
in a normalised quantum state vector $\psi \in \sF$, that is $\|\psi\|_{\sF} = 1$.

We consider $\gotC_1 = \gotC_1(\sF)$, the complex Banach space of trace-class operators on
$\sF$ with trace-norm $\|\cdot\|_{\gotC_1}$.
Its dual space is isometrically isomorphic to the Banach space of
all bounded operators $\cl(\sF)$.
The corresponding dual pair is determined by the bilinear trace functional
\begin{equation} \label{dual-Tr}
\langle \phi \mid  A \rangle_{\gotC_1(\sF) \times \cl(\sF)} = \Tr (\phi \, A) ,
\end{equation}
where $\phi \in \gotC_1(\sF)$ and  $A \in \cl(\sF)$.

The quantum-mechanical Hamiltonian evolution of the isolated system (\ref{h-E}) is determined
by the unitary group $(U_{ih}(t))_{t \in \Ri}$, where
$U_{ih}(t) = e^{- i t h} \in \cl(\sF)$ for all $t \in \Ri$.
For all $t \in (0,\infty)$ define $W_t \colon \gotC_1^{\rm sa} \to \gotC_1^{\rm sa}$ by
\begin{equation}\label{I-SG}
W_t \rho = U_{ih}(t) \, \rho \, U_{ih}(t)^* .
\end{equation}
Then $(W_t)_{t > 0}$ is evidently a contraction $C_0$-semigroup, which is positivity preserving
and trace preserving.
The semigroup $(W_t)_{t > 0}$ is called
the Markov dynamical (semi)group for the evolution of the
isolated system (\ref{h-E}).
Let $-L$ be the generator of $(W_t)_{t > 0}$.
Define $\Psi \colon \gotC_1^{\rm sa} \to \gotC_1^{\rm sa}$  by
\[
\Psi(\rho) = (I + \hat n )^{-1} \, \rho \, (I + \hat n )^{-1}
.  \]
Then $\Psi(\gotC_1^{\rm sa}) \subset D(L)$ and
\[
L \Psi(\rho_0)
= i \, h \, (I + \hat n )^{-1} \, \rho_0 \, (I + \hat n )^{-1}
   - i \, (I + \hat n )^{-1} \, \rho_0 \, h \, (I + \hat n )^{-1}
\]
for all $\rho_0 \in \gotC_1^{\rm sa}$.
Note that
\[
L \rho \supset i \, [h, \rho]
\]
for all $\rho \in \Psi(\gotC_1^{\rm sa})$.

To illustrate an \textit{open} system corresponding to (\ref{h-E}), we consider the simplest
model when this system is in contact with an external reservoir of bosons $b, b^*$.
Then to describe the evolution of this open system
we follow the Kossakowski--Lindblad--Davies (KLD) Ansatz for the dissipative extension of the
Hamiltonian positivity preserving dynamics (\ref{I-SG}) to a non-Hamiltonian
positivity preserving evolution.

Fix $\sigma_{\pm} \in [0,\infty)$.
Define the operator $\cQ \colon D(\cQ) \to \gotC_1^{\rm sa}$
with domain $D(\cQ) = \Psi(\gotC_1^{\rm sa})$ by
\begin{equation}\label{eqK-L-D1}
\cQ \rho
= \sigma_- \, \Big( b \, (I + \hat n )^{-1/2} \Big)  \, \rho_0 \,
              \Big( b \, (I + \hat n )^{-1/2} \Big)^*
    + \sigma_+ \, \Big( b^*\, (I + \hat n )^{-1/2} \Big) \rho_0  \,
                  \Big( b^*\, (I + \hat n )^{-1/2} \Big)^*,
\end{equation}
where $\rho_0 \in \gotC_1^{\rm sa}$ is such that $\rho = \Psi(\rho_0)$.
Note that
\[
\cQ \rho \supset \sigma_- \, b \, \rho \, b^* + \sigma_+ \, b^*\, \rho \, b
\]
and that the operator $\rho \mapsto \cQ \Psi(\rho)$ is continuous
from $\gotC_1^{\rm sa}$ into $\gotC_1^{\rm sa}$.
Since $\hat n$ is densely defined, it is not hard to show that
$\Psi(\gotC_1^{\rm sa}) \cap \gotC_1^+ = \Psi(\gotC_1^+)$.
Hence $\cQ$ is a positivity preserving operator.

Using the bilinear trace functional (\ref{dual-Tr}), the
dual operator $\cQ^*$ acts in $\cl(\sF)$.
It is defined via the relation
$\langle \cQ\,\rho \mid  A \rangle_{\gotC_1(\sF) \times \cl(\sF)}
= \langle \rho \mid \cQ^*(A) \rangle_{\gotC_1(\sF) \times \cl(\sF)}$.
If $A_0 \in \cl(\sF)$ and $A = (I + \hat n )^{-1/2} \, A_0 \, (I + \hat n )^{-1/2}$,
then $A \in D(\cQ^*)$ and
\[
\cQ^*(A) \supset \sigma_- \, b^* \, A \, b + \sigma_+ \, b \, A \, b^*
.  \]
If $\sigma_+ + \sigma_- > 0$ then clearly $I \not\in D(\cQ^*)$.

The non-Hamilton\-ian evolution equation
$\partial_t\rho(t) = - \widetilde L_\sigma \rho(t)$
is defined formally in the framework of the KLD Ansatz
with the generator $- \widetilde L_\sigma$, where
\[
\widetilde L_\sigma \rho = i \, [h, \rho] +
\frac{1}{2} (\cQ^*(I) \rho + \rho \, \cQ^* (I))  - \, \cQ \rho
\]
and formally $\cQ^*(I) = \sigma_- \, b^* \, b + \sigma_+ \, b \, b^*$.
Therefore formally
\begin{equation}\label{K-L-D-Generator}
\widetilde L_\sigma \rho
= i \, [h, \rho] +
\frac{1}{2} \Big( (\sigma_- \, b^* \, b + \sigma_+ \, b \, b^*) \rho
                    + \rho \, (\sigma_- \, b^* \, b + \sigma_+ \, b \, b^*) \Big)  - \, \cQ \rho .
\end{equation}
We aim to give a mathematical sense of (\ref{K-L-D-Generator}) and to define
the corresponding semigroup.

To proceed we first consider the operator
$h_\sigma \colon D(\hat n) \to \gotC_1^{\rm sa}$  defined by
\[
h_\sigma  = i \, h + \tfrac{1}{2} \, (\sigma_- \, b^* \, b + \sigma_+ \, b \, b^*)
.  \]
Then $h_\sigma$ is an $m$-accretive operator.
Define $U_{h_\sigma }(t) = e^{- t \, h_\sigma } \in \cl(\sF)$ for all $t \in [0,\infty)$.
Then similarly to (\ref{I-SG}) the contraction $C_0$-semigroup $(U_{h_\sigma }(t))_{t > 0}$ induces
on the Banach space $\gotC_1^{\rm sa}$ a positivity preserving contraction $C_0$-semigroup
$(S^\sigma_t)_{t > 0}$ given by
\[
S^\sigma_t\rho = U_{h_\sigma }(t) \, \rho \, U_{h_\sigma }(t)^* .
\]
Let $-H_\sigma$ be the generator of the semigroup $(S^\sigma_t)_{t > 0}$.
Then
$ D(H_\sigma) \supset \Psi(\gotC_1^{\rm sa})$.
If $\rho \in \Psi(\gotC_1^{\rm sa})$, then
\begin{equation}\label{h-sigma}
H_\sigma \rho
\supset i \, [h, \rho]
   + \frac{1}{2} \Big( (\sigma_- \, b^* \, b + \sigma_+ \, b \, b^*) \rho
                    + \rho \, (\sigma_- \, b^* \, b + \sigma_+ \, b \, b^*) \Big) .
\end{equation}
Moreover, the map $\rho \mapsto H_\sigma \Psi(\rho)$ is continuous
from $\gotC_1^{\rm sa}$ into $\gotC_1^{\rm sa}$.
Also, if $\rho \in \Psi(\gotC_1^+)$, then $\Tr H_\sigma  \rho \geq 0$.
Since $S^\sigma_t$ commutes with the operator
$\Psi$, one deduces that
\[
S^\sigma_t \Psi(\gotC_1^{\rm sa}) \subset \Psi(\gotC_1^{\rm sa}) .
\]
Hence $\Psi(\gotC_1^{\rm sa})$ is a core for operator $H_\sigma$.

Note that whenever $\sigma_- + \sigma_+ > 0$, the semigroup $(S^\sigma_t )_{t > 0}$ is \textit{not}
trace-preserving.
Indeed, if $\rho \in \gotC_1^+$ is given by
$\rho(\varphi) = (\varphi,e_1)_{\sF} \, e_1$, then
$H_\sigma \rho = (\sigma_- + 2 \sigma_+) \rho$.
Hence $S^\sigma_t \rho = e^{- (\sigma_- + 2 \sigma_+) t} \rho$ and
$\Tr (S^\sigma_t \rho) = e^{- (\sigma_- + 2 \sigma_+) t}$ for all $t > 0$.

\begin{remark} \label{rkg320}
The operator $H_\sigma$ is not positivity preserving, even although the semigroup $(S^\sigma_t)_{t > 0}$
is positivity preserving.
An example is as follows.
For simplicity assume that $E = 1$.
Using the commutation relation $(b \, b^* - b^* \, b) \psi = \psi$ for all $\psi \in D(b^* \, b)$,
one deduces that
\[
H_\sigma \rho
\supset i \, (\hat n \, \rho - \rho \, \hat n)
   + \frac{1}{2} (\sigma_- + \sigma_+) (\hat n \, \rho + \rho \, \hat n)
   + \sigma_+ \, \rho
\]
for all $\rho \in \Psi(\gotC_1^{\rm sa})$.
Let $k \in \Ni$ and $\lambda > 0$.
Choose $\psi = e_1 + i \, \lambda \, e_k$.
Define $\rho \in \Psi(\gotC_1^{\rm sa})^+$ by $\rho(\varphi) = (\varphi,\psi)_\sF \, \psi$.
Then
\begin{eqnarray*}
(H_\sigma \rho) \varphi
& = & i \Big( (\varphi,\psi)_\sF \, \hat n \, \psi - (\hat n \, \varphi,\psi)_\sF \, \psi \Big)  \\*
& & \hspace*{10mm} {}
   + (\sigma_- + \sigma_+)
        \Big( (\varphi,\psi)_\sF \, \hat n \, \psi + (\hat n \, \varphi,\psi)_\sF \, \psi \Big)
   + \sigma_+ \, (\varphi,\psi)_\sF \, \psi
\end{eqnarray*}
for all $\varphi \in D(\hat n)$.
So
\begin{eqnarray*}
((H_\sigma \rho) \varphi, \varphi)_\sF
& = & - 2 \IIm \Big( (\varphi,\psi)_\sF \, (\hat n \, \psi,\varphi)_\sF \Big) \\*
& & \hspace*{10mm} {}
      + (\sigma_- + \sigma_+)
           \RRe \Big( (\varphi,\psi)_\sF \, (\hat n \, \psi,\varphi)_\sF \Big)
   + \sigma_+ \, |(\varphi,\psi)_\sF|^2
.
\end{eqnarray*}
Now choose $\varphi = e_1 + e_k$.
Then
$(\varphi,\psi)_\sF \, (\hat n \, \psi,\varphi)_\sF
= (1-i \, \lambda)(1 + i \, k \, \lambda)
= 1 + k \, \lambda^2 + i \, (k-1) \, \lambda$.
So
\[
((H_\sigma \rho) \varphi, \varphi)_\sF
= - 2 (k-1) \, \lambda
   + (\sigma_- + \sigma_+) (1 + k \, \lambda^2)
   + \sigma_+ \, (1 + \lambda^2)
.  \]
Choose $\lambda > 0$ such that $\lambda \, (\sigma_- + \sigma_+) < 1$.
Then $((H_\sigma \rho) \varphi, \varphi)_\sF < 0$ for large $k \in \Ni$.
Therefore the operator $H_\sigma \rho$ is not positive and the
operator $H_\sigma$ is not positivity preserving.
\end{remark}

\subsection{A particle-number cut-off regularisation }\label{RvPNC-O}

To make precise the meaning of the operator formally introduced
in (\ref{K-L-D-Generator}) we use
(\ref{h-sigma}) and the next two lemmata for an extension of~$\cQ$.
The first lemma is about boundedness of operators.

\begin{lemma} \label{lkato300.5}
Let $A \colon \Psi(\gotC_1^{\rm sa}) \to \gotC_1^{\rm sa}$ be a positivity
preserving operator and assume that $\Tr A \rho \leq \Tr \rho$ for all $\rho \in \Psi(\gotC_1^+)$.
Then $\|A \rho\|_{\gotC_1} \leq \|\rho\|_{\gotC_1}$ for all
$\rho \in \Psi(\gotC_1^{\rm sa})$.
\end{lemma}
\begin{proof}
\firststep \label{lkato300.5:1}
Let $\rho \in \Psi(\gotC_1^{\rm sa})$ and $\varepsilon > 0$.
We first show that there exist $\rho_1,\rho_2 \in \Psi(\gotC_1^+)$
such that $\rho = \rho_1 - \rho_2$ and
$\Tr \rho_1 + \Tr \rho_2 \leq \|\rho\|_{\gotC_1} + \varepsilon$.

The proof is a modification of \cite{Dav17} Lemma~2.1.
By assumption there exists a $\rho_0 \in \gotC_1^{\rm sa}$ such that
$\rho = (I + \hat n)^{-1} \, \rho_0 \, (I + \hat n)^{-1}$.
For all $t > 0$ define
\[
\rho_t
= (I + t \, \hat n) \, (I + \hat n)^{-1} \, \rho_0 \, (I + t \, \hat n) \, (I + \hat n)^{-1}
.  \]
Then $\rho_t \in \gotC_1^{\rm sa}$.
Moreover, $\lim_{t \downarrow 0} \rho_t = \rho$ in $\gotC_1$.
Hence there exists a $t > 0$ such that
$\|\rho_t\|_{\gotC_1} \leq \|\rho\|_{\gotC_1} + \varepsilon$.

By Lemma~\ref{lkg200.5}\ref{lkg200.5-1} there exist $v,w \in \gotC_1^+$ such that
$\rho_t = v-w$ and
$|\rho_t| = v+w$.
Write
\[
\rho_1 = (I + t \, \hat n)^{-1} \, v \, (I + t \, \hat n)^{-1}
\quad \mbox{and} \quad
\rho_2 = (I + t \, \hat n)^{-1} \, w \, (I + t \, \hat n)^{-1}
.  \]
Then $\rho_1,\rho_2 \in \Psi(\gotC_1^+)$ and
\[
\rho_1 - \rho_2
= (I + t \, \hat n)^{-1} \, (v-w) \, (I + t \, \hat n)^{-1}
= (I + t \, \hat n)^{-1} \, \rho_t \, (I + t \, \hat n)^{-1}
= (I + \hat n)^{-1} \, \rho_0 \, (I + \hat n)^{-1}
= \rho
.  \]
Moreover,
\begin{eqnarray*}
\Tr \rho_1 + \Tr \rho_2
& = & \Tr \Big( (I + t \, \hat n)^{-1} \, (v+w) \, (I + t \, \hat n)^{-1} \Big) \\
& = & \|(I + t \, \hat n)^{-1} \, (v+w) \, (I + t \, \hat n)^{-1}\|_{\gotC_1}  \\
& \leq & \|v+w\|_{\gotC_1}
= \| \, |\rho_t| \, \|_{\gotC_1}
= \|\rho_t\|_{\gotC_1}
\leq \|\rho\|_{\gotC_1} + \varepsilon
\end{eqnarray*}
as required.

\nextstep\
Now we prove the lemma.
Let $\rho \in \Psi(\gotC_1^{\rm sa})$.
Let $\varepsilon > 0$ and let $\rho_1,\rho_2 \in \Psi(\gotC_1^+)$
be as in Step~\ref{lkato300.5:1}.
Then
\[
\|A \rho\|_{\gotC_1}
\leq \|A \rho_1\|_{\gotC_1} + \|A \rho_2\|_{\gotC_1}
= \Tr A \rho_1 + \Tr A \rho_2
\leq \Tr \rho_1 + \Tr \rho_2
\leq \|\rho\|_{\gotC_1} + \varepsilon
\]
and the lemma follows.
\end{proof}

\begin{lemma} \label{lkato301}
The operator $\cQ$ extends uniquely
to a continuous operator $\widehat \cQ \colon D(H_\sigma) \to \gotC_1^{\rm sa}$,
where $D(H_\sigma)$ is provided with the graph norm.
Moreover, $\widehat \cQ$ is positivity preserving,
\[
\Tr (H_\sigma \rho - \widehat \cQ \rho) = 0
\]
and $\|\widehat \cQ \rho\|_{\gotC_1} \leq \|H_\sigma \rho\|_{\gotC_1}$
for all $\rho \in D(H_\sigma)$.
\end{lemma}
\begin{proof}
Note that
\begin{equation} \label{rel-bound-1}
\Tr (H_\sigma  \rho - \cQ \rho) = 0
\end{equation}
for all $\rho \in \Psi(\gotC_1^{\rm sa})$.
Let $\lambda > 0$.
The resolvent
\[
(\lambda \, I + H_\sigma)^{-1} = \int_0^\infty e^{- t \, \lambda} \, S^\sigma_t \, dt
\]
is positivity preserving and
$(\lambda \, I + H_\sigma)^{-1} \Psi(\gotC_1^{\rm sa}) \subset \Psi(\gotC_1^{\rm sa})$
since $S^\sigma_t$ commutes with the operator $\Psi$.
Then the map
$\cQ (\lambda \, I + H_\sigma)^{-1}|_{\Psi(\gotC_1^{\rm sa})}
      \colon \Psi(\gotC_1^{\rm sa}) \to  \gotC_1^{\rm sa}$ is
also positivity preserving.
Moreover, (\ref{rel-bound-1}) yields
\begin{eqnarray*}
\Tr \Big( \cQ (\lambda \, I + H_\sigma)^{-1} \rho \Big)
& = & \Tr \Big( H_\sigma (\lambda \, I + H_\sigma)^{-1} \rho \Big)  \nonumber \\
& = & \Tr \rho - \lambda \Tr (\lambda \, I + H_\sigma)^{-1} \rho
\leq \Tr \rho
\end{eqnarray*}
for all $\rho \in \Psi(\gotC_1^+) = \Psi(\gotC_1^{\rm sa}) \cap \gotC_1^+$.
Hence the operator $\cQ (\lambda \, I + H_\sigma)^{-1}|_{\Psi(\gotC_1^{\rm sa})}$
is bounded by Lemma~\ref{lkato300.5}, with norm at most~$1$.
Since $\Psi(\gotC_1^{\rm sa})$ is dense in $\gotC_1^{\rm sa}$ one deduces that
the operator $\cQ (\lambda \, I + H_\sigma)^{-1}|_{\Psi(\gotC_1^{\rm sa})}$
has a unique bounded extension $E_\lambda \colon  \gotC_1^{\rm sa} \to \gotC_1^{\rm sa}$,
which is a positivity preserving operator.
Then $\|E_\lambda\| \leq 1$.

Define the operator $\widehat \cQ_\lambda \colon D(H_\sigma) \to \gotC_1^{\rm sa}$
by
\[
\widehat \cQ_\lambda
= E_\lambda (\lambda \, I + H_\sigma)
.  \]
Then
$\|\widehat \cQ_\lambda \rho\|_{\gotC_1} \leq \|(\lambda \, I + H_\sigma) \rho\|_{\gotC_1}$
for all $\rho \in D(H_\sigma)$.
So $\widehat \cQ_\lambda$ is continuous from $D(H_\sigma)$ into~$\gotC_1^{\rm sa}$.

Since $\Psi \circ S^\sigma_t = S^\sigma_t \circ \Psi$ for all $t > 0$,
it follows that $\Psi(\rho) \in D(H_\sigma)$ and
$H_\sigma \Psi(\rho) = \Psi(H_\sigma \rho)$ for all $\rho \in D(H_\sigma)$.
If $\rho \in D(H_\sigma)$, then
\begin{eqnarray}
\widehat \cQ_\lambda \Psi(\rho)
& = & E_\lambda (\lambda \, I + H_\sigma) \Psi(\rho)  \nonumber  \\
& = & E_\lambda \Psi((\lambda \, I + H_\sigma) \rho)
= \cQ \, (\lambda \, I + H_\sigma)^{-1} \Psi((\lambda \, I + H_\sigma) \rho)
= \cQ \Psi(\rho)
.
\label{elkato301;1}
\end{eqnarray}
The map $\rho \mapsto \Psi(\rho)$ is continuous from $\gotC_1^{\rm sa}$ into $D(H_\sigma)$
and $\widehat \cQ_\lambda$ is continuous from $D(H_\sigma)$ into $\gotC_1^{\rm sa}$.
So $\rho \mapsto \widehat \cQ_\lambda \Psi(\rho)$ is continuous from
$\gotC_1^{\rm sa}$ into $\gotC_1^{\rm sa}$.
Also $\rho \mapsto \cQ \Psi(\rho)$ is continuous from
$\gotC_1^{\rm sa}$ into $\gotC_1^{\rm sa}$.
Hence it follows from (\ref{elkato301;1}) that
$\widehat \cQ_\lambda \Psi(\rho) = \cQ \Psi(\rho)$
for all $\rho \in \gotC_1^{\rm sa}$.
In particular, $\widehat \cQ_\lambda$ is an extension of $\cQ$.
Since $\Psi(\gotC_1^{\rm sa})$ is dense in $D(H_\sigma)$,
it follows that $\widehat \cQ_\lambda$ is the unique continuous operator
from $D(H_\sigma)$ into $\gotC_1^{\rm sa}$ which extends
$\cQ$.
Consequently $\widehat \cQ_\lambda$ is independent of $\lambda$ and we set
$\widehat \cQ = \widehat \cQ_1$.

If $\rho \in D(H_\sigma)$, then
\[
\|\widehat \cQ \rho\|_{\gotC_1}
= \|\widehat \cQ_\lambda \rho\|_{\gotC_1}
\leq \|(\lambda \, I + H_\sigma) \rho\|_{\gotC_1}
\leq \lambda \, \|\rho\|_{\gotC_1} + \|H_\sigma \rho\|_{\gotC_1}
\]
for all $\lambda > 0$.
So $\|\widehat \cQ \rho\|_{\gotC_1} \leq \|H_\sigma \rho\|_{\gotC_1}$.

It follows from (\ref{rel-bound-1}) that
$\Tr (\widehat \cQ \Psi(\rho))
= \Tr (\cQ \Psi(\rho))
= \Tr (H_\sigma \Psi(\rho))$
for all $\rho \in \gotC_1^{\rm sa}$.
Then by density and continuity $\Tr (\widehat \cQ \rho) = \Tr (H_\sigma \rho)$
for all $\rho \in D(H_\sigma)$.

It remains to show that $\widehat{\cQ}$ is positivity preserving.
Let $\psi \in \sF$ and $t > 0$.
If $\rho \in \gotC_1^+$, then $S^\sigma_t \rho \in \gotC_1^+$ and
\[
( (\widehat{\cQ} \, S^\sigma_t \, \Psi(\rho))\psi, \psi)_{\sF}
= ( (\widehat{\cQ} \, \Psi(S^\sigma_t \rho))\psi, \psi)_{\sF}
= ( (\cQ \, \Psi(S^\sigma_t \rho))\psi, \psi)_{\sF}
\geq 0
\]
since $\cQ$ is positivity preserving.
Because $\Psi(\gotC_1^+)$ is dense in $\gotC_1^+$, one deduces
that
\[
( (\widehat{\cQ} \, S^\sigma_t \rho)\psi, \psi)_{\sF} \geq 0
\]
for all $\rho \in \gotC_1^+$.
Now let $\rho \in D(H_\sigma)^+$.
Then
$( (\widehat{\cQ} \rho)\psi, \psi)_{\sF}
= \lim_{t \downarrow 0} ( (\widehat{\cQ} \, S^\sigma_t \rho)\psi, \psi)_{\sF}
\geq 0$.
Therefore $\widehat{\cQ}$ is positivity preserving.
\end{proof}

Let $\widehat \cQ$ be as in Lemma~\ref{lkato301}.
Since there will be no confusion, we will denote $\widehat \cQ$ by $\cQ$.

We shall use the general approach developed in Section~\ref{Skg2}.
To this aim we consider a regularisation generated by the family of \textit{projections}
$(P_N)_{N \in \Ni_0}$, where for all $N \in \Ni_0$ the projection
$P_N \colon \sF \to \sF$ is given by
\[
P_N \psi := \sum_{n=0}^N (\psi , e_n)_{\sF} \, e_n  .
\]
Note that the number of bosons in the subspace $P_N \sF$ is bounded because
the boson number operator satisfies $\|\hat n  (P_N  \psi)\|_{\sF} \leq N \, \|\psi\|_{\sF}$
for all $\psi \in \sF$.

Obviously
$\lim_{N \to \infty} P_N \psi = \psi$ for all $\psi \in \sF$.
For all $N \in \Ni_0$ define the \textit{particle number cut-off} regularisation
$\cQ_N  \in \cl(\gotC_1^{\rm sa})$
of the operator $\cQ$ by
\begin{equation}\label{reg-0}
\cQ_N \rho
= \sigma_- \, (b^* \, P_N )^* \, \rho \, (b^* \, P_N)
   + \sigma_+ \, (b \, P_N)^* \, \rho \, (b \, P_N ) .
\end{equation}
Note that $\cQ_N  \rho =  P_N \, ( \cQ \rho) \, P_N$
for all $\rho \in \Psi(\gotC_1^{\rm sa})$ by (\ref{eqK-L-D1}).
Therefore $\|\cQ_N \rho\|_{\gotC_1} \leq \|\cQ \rho\|_{\gotC_1}$
for all $\rho \in \Psi(\gotC_1^{\rm sa})$ and then by density
$\|\cQ_N \rho\|_{\gotC_1} \leq \|\cQ \rho\|_{\gotC_1}$
for all $\rho \in D(H_\sigma)$.

We next verify that $(\cQ_N)_{N \in \Ni_0}$ is a functional regularisation of $\cQ$.
Clearly $\cQ_N$ is positivity preserving for all $N \in \Ni_0$, which is
Condition~\ref{dkg102-1} in Definition~\ref{dkg102}.
The definition of $\cQ_N$ implies the estimate
\[
\|\cQ_N  \rho\|_{\gotC_1}
\leq (\sigma_-(N+1) + \sigma_+ N) \, \|\rho\|_{\gotC_1} ,
\]
for all $\rho \in \gotC_1^{\rm sa}$, which implies Definition~\ref{dkg102}\ref{dkg102-2}.
Since $\sigma_{\pm} \geq 0$, the regularisation (\ref{reg-0}) is
monotone increasing as
a sequence of positivity preserving maps in $\gotC_1^{\rm sa}$, and bounded by $\cQ$.
So Condition~\ref{dkg102-3} in Definition~\ref{dkg102} is valid.
Finally we show that
$\lim_{N \to \infty} ((\cQ_N \rho) \psi, \psi)_{\sF} = ((\cQ \rho) \psi, \psi)_{\sF}$
for all $\rho \in D(H_\sigma)$ and $\psi \in \sF$.
Let $\psi \in \sF$.
Let $\rho \in \Psi(\gotC_1^{\rm sa})$.
Then
\[
\lim_{N \to \infty} ((\cQ_N \rho) \psi, \psi)_{\sF}
= \lim_{N \to \infty} ((\cQ \rho) \, P_N \psi, P_N \psi)_{\sF}
= ((\cQ \rho) \psi, \psi)_{\sF}
\]
for all $\rho \in \Psi(\gotC_1^{\rm sa})$.
Since $\Psi(\gotC_1^{\rm sa})$ is dense in $D(H_\sigma)$ and
$\|\cQ_N \rho\|_{\gotC_1} \leq \|\cQ \rho\|_{\gotC_1}$
for all $\rho \in D(H_\sigma)$ and $N \in \Ni_0$, one deduces that
$\lim_{N \to \infty} ((\cQ_N \rho) \psi, \psi)_{\sF} = ((\cQ \rho) \psi, \psi)_{\sF}$
for all $\rho \in D(H_\sigma)$ and $\psi \in \sF$.
So $(-\cQ_N)_{N \in \Ni_0}$ satisfies Definition~\ref{dkg102}\ref{dkg102-4}.

We proved that the family $(\cQ_N)_{N \in \Ni_0}$ is a
functional regularisation of the
operator $\cQ$.
For all $N \in \Ni$ define the operator $L_{\sigma,N}$ by
\[
L_{\sigma,N} = H_\sigma  - \cQ_N
\]
with domain $D(L_{\sigma,N}) = D(H_\sigma)$.
Let $(T_{t,N}^\sigma)_{t > 0}$ be the semigroup generated by $-L_{\sigma,N}$.
Then it follows from Theorem~\ref{tkg101} that
$(T_{t,N}^\sigma)_{t > 0}$
is a positivity preserving contraction semigroup, so it is a dynamical semigroup.
Moreover, for all $t > 0$ and $\rho \in \gotC_1^{\rm sa}$ the limit
\[
T_t^\sigma \rho
= \lim_{N \to \infty} T_{t,N}^\sigma \rho
\]
exists in $\gotC_1$ and $(T_t^\sigma)_{t > 0}$
is a positivity preserving contraction $C_0$-semigroup on $\gotC_1^{\rm sa}$.
Let $-L_\sigma$ be the generator of $(T_t^\sigma)_{t > 0}$.
Then $L_\sigma$ is an extension of the operator $H_\sigma - \cQ$.
By Theorem~\ref{tkg204} the semigroup $(T_t^\sigma)_{t > 0}$ is minimal in the
following sense:
If $(\widehat T_t^\sigma)_{t > 0}$ is a positivity preserving $C_0$-semigroup
with generator $- \widehat L_\sigma$,
which is an extension of $- (H_\sigma - \cQ)$,
then $\widehat T_t^\sigma \geq T_t^\sigma$ for all $t > 0$.

\subsection{Core property and trace-preserving}

A priori it is unclear whether the (minimal) dynamical semigroup $(T^\sigma_t)_{t > 0}$
is trace-preserving (and hence is a Markov dynamical semigroup).
We know that $\Tr (H_\sigma \rho - \cQ \rho) = 0$ for all $\rho \in D(H_\sigma)$
by Lemma~\ref{lkato301}.
Therefore if $D(H_\sigma)$ is a core for $L_\sigma$, then we can use
Theorem~\ref{tkg205} to conclude that the semigroup $(T^\sigma_t)_{t > 0}$
is trace-preserving.
We shall show that this is the case if $\sigma_+ < \sigma_-$.

\begin{thm} \label{D(L)-coreM}
If $\sigma_+ < \sigma_-$, then the domain $D(H_\sigma)$ is a core for $L_\sigma$.
\end{thm}
\begin{proof}
Fix $s \in (0, \infty)$ such that $\sigma_+ \, e^{2s} < \sigma_-$.
Define the map $R \colon \gotC_1^{\rm sa} \to \gotC_1^{\rm sa}$ by
\[
R \rho := e^{-s \hat n} \, \rho \, e^{-s \hat n} .
\]
Then $R$ is a  positivity preserving contraction and
$R(\gotC_1^{\rm sa})\subset \Psi(\gotC_1^{\rm sa})$.
If $t > 0$, then $S^\sigma_t$ and $R$ commute.
Hence if $\rho \in D(H_\sigma)$, then $R \rho \in D(H_\sigma)$ and
\begin{equation}\label{R-s3}
H_\sigma \, R  \rho = R \, H_\sigma \rho .
\end{equation}

Let $\cQ_-$ and $\cQ_+$ be the positive operators in $\gotC_1^{\rm sa}$ with
domain $D(\cQ_-) = D(\cQ_+) = \Psi(\gotC_1^{\rm sa})$, defined similarly as in
(\ref{eqK-L-D1}) such that
\[
\cQ_- \rho
\supset \sigma_- \, b \, \rho \, b^*
\quad \mbox{and} \quad
\cQ_+ \rho
\supset \sigma_+ \, b^* \, \rho \, b
.  \]
Then $\Tr (\cQ_- \rho) \leq \Tr (\cQ \rho) = \Tr (H_\sigma \rho)$ and
$\Tr (\cQ_+ \rho) \leq \Tr (H_\sigma \rho)$
for all $\rho \in \Psi(\gotC_1^+)$.
Arguing as in Lemma~\ref{lkato301} one deduces that that there exist unique continuous
extensions of $\cQ_+$
and of $\cQ_-$, also denoted by
$\cQ_+$ and $\cQ_-$, with domain $D(H_\sigma)$,
such that
$\|\cQ_+ \rho\|_{\gotC_1} \leq \|H_\sigma \rho\|_{\gotC_1}$
and $\|\cQ_- \rho\|_{\gotC_1} \leq \|H_\sigma \rho\|_{\gotC_1}$ for all $\rho \in D(H_\sigma)$.
Then $\cQ = \cQ_- + \cQ_+$.
Note that
\[
b \, e^{-s \hat n} \rho
= e^{-s} \, e^{-s \hat n} \, b \rho
\quad \mbox{and} \quad
e^{-s \hat n} \, b^* \rho
= e^{-s} \, b^* \, e^{-s \hat n} \rho
\]
for all $\rho \in D(\hat n)$.
Therefore
\[
\cQ_- \, R \rho = e^{-2s} R \, \cQ_- \rho
\quad \mbox{and} \quad
\cQ_+ \, R \rho = e^{2s} R \, \cQ_+ \rho
\]
first for all $\rho \in \Psi(\gotC_1^{\rm sa})$ and
then by density for all $\rho \in D(H_\sigma)$.
Together with (\ref{R-s3}) this implies that
\begin{equation} \label{L-R-L-bis}
(H_\sigma - \cQ) \, R \rho
= R \, (H_\sigma  - e^{-2s} \, \cQ_- - e^{2s} \, \cQ_+) \rho
= R \, (H_\sigma  - \widetilde{\cQ}) \rho
\end{equation}
for all $\rho \in D(H_\sigma)$,
where the positivity preserving operator
$\widetilde{\cQ} \colon D(H_\sigma) \to \gotC_1^{\rm sa}$ is defined by
\[
\widetilde{\cQ}
= e^{-2s} \, \cQ_- + e^{2s} \, \cQ_+
.  \]
Define
\[
r = \frac{e^{-2s} \, \sigma_- + e^{2s} \, \sigma_+}{\sigma_- + \sigma_+} .
\]
Then $r \in (0,1)$ since $s > 0$ and $\sigma_+ \, e^{2s} < \sigma_-$.
Moreover,
$e^{-2s} - r = - \frac{2 \sigma_+ \sinh 2s}{\sigma_- + \sigma_+}$
and $e^{2s} - r = \frac{2 \sigma_- \sinh 2s}{\sigma_- + \sigma_+}$.
Therefore
\begin{eqnarray*}
\Tr (\widetilde \cQ \rho)
& = & e^{-2s} \, \sigma_- \Tr (b^* \, b \, \rho)
     + e^{2s} \, \sigma_+ \Tr (b \, b^* \, \rho)   \\
& = & r \, \sigma_- \Tr (b^* \, b \, \rho) + r \, \sigma_+ \Tr (b \, b^* \, \rho)
   + (e^{-2s} - r) \, \sigma_- \Tr (b^* \, b \, \rho)
   + (e^{2s} - r) \, \sigma_+ \Tr (b \, b^* \, \rho)  \\
& = & r \Tr (H_\sigma \rho)
   + \frac{2 \sigma_- \, \sigma_+ \, \sinh 2s}{\sigma_- + \sigma_+} \,
        \Tr ((b \, b^* - b^* \, b) \, \rho)  \\
& = & r \Tr (H_\sigma \rho)
      + \frac{2 \sigma_- \, \sigma_+ \, \sinh 2s}{\sigma_- + \sigma_+} \, \Tr \rho
\end{eqnarray*}
for all $\rho \in \Psi(\gotC_1^{\rm sa})$,
where we use the canonical commutation relations in the last step.
Hence
\[
\Tr (\widetilde \cQ \rho)
= r \Tr (H_\sigma \rho)
      + \frac{2 \sigma_- \, \sigma_+ \, \sinh 2s}{\sigma_- + \sigma_+} \, \Tr \rho
\]
for all $\rho \in D(H_\sigma)$ by density and continuity.

It follows from Proposition~\ref{pkg230}\ref{pkg230-1} that the operator
$H_\sigma - \widetilde \cQ$ is quasi-$m$-accretive.
Hence there exists a $\lambda > 0$ such that $\lambda \, I + H_\sigma - \widetilde \cQ$ is
invertible.
Since $L_\sigma$ is an extension of $H_\sigma - \cQ$ it follows from (\ref{L-R-L-bis})
that
\[
(\lambda \, I + L_\sigma) R \rho
= (\lambda \, I + H_\sigma - \cQ) R \rho
= R (\lambda \, I + H_\sigma - \widetilde \cQ) \rho
\]
for all $\rho \in D(H_\sigma)$.
Hence
\[
(\lambda \, I + L_\sigma) (D(H_\sigma))
\supset (\lambda \, I + L_\sigma) R (D(H_\sigma))
= R (\lambda \, I + H_\sigma - \widetilde \cQ) (D(H_\sigma))
= R (\gotC_1^{\rm sa})
\]
is dense in $\gotC_1^{\rm sa}$.
Consequently $D(H_\sigma)$ is dense in $D(L_\sigma)$, that is
$D(H_\sigma)$ is a core for $L_\sigma$.
\end{proof}

\begin{cor} \label{ckg320}
If $\sigma_+ < \sigma_-$, then the semigroup $(T^\sigma_t)_{t > 0}$ is
trace-preserving.
\end{cor}
\begin{proof}
This follows from Theorem~\ref{tkg205}, Lemma~\ref{lkato301} and
Theorem~\ref{D(L)-coreM}.
\end{proof}

\begin{cor} \label{cPsiX-coreM}
If $\sigma_+ < \sigma_-$, then the set $\Psi(\gotC_1^{\rm sa})$ is a core
for the operator $L_\sigma$.
\end{cor}
\begin{proof}
The set
$\Psi(\gotC_1^{\rm sa})$ is dense in $D(H_\sigma)$.
Moreover, $D(H_\sigma)$ is dense in $D(L_\sigma)$ by Theorem~\ref{D(L)-coreM}.
Hence $\Psi(\gotC_1^{\rm sa})$ is dense in $D(L_\sigma)$,
that is $\Psi(\gotC_1^{\rm sa})$ is a core for the operator~$L_\sigma$.
\end{proof}

The proof of Theorem~\ref{D(L)-coreM} is heavily based on the
strict inequality $\sigma_+ < \sigma_-$.
We do not know whether $D(H_\sigma)$ is a core for $L_\sigma$
if $\sigma_+ = \sigma_- > 0$.

\medskip

We comment that an alternative regularisation for $\cQ$ is possible.
For all $N \in \Ni_0$ define $\checkQ_N \in \cl(\gotC_1^{\rm sa})$
by
\[
\checkQ_N \rho
= \cQ(P_N \, \rho \, P_N)
.  \]
It is easy to verify that $(\checkQ_N)_{N \in \Ni_0}$ satisfies
Conditions~\ref{dkg102-1}, \ref{dkg102-2} and \ref{dkg102-3} in
Definition~\ref{dkg102}.
We next verify Condition~\ref{dkg102-4}.
Choose $V = D(b) = D(b^*)$.
Then $V$ is dense in $\sF$.
Let $\psi \in V$.
If $\rho \in \Psi(\gotC_1^{\rm sa})$ and $N \in \Ni_0$, then
$((\checkQ_N \rho) \psi, \psi)_{\sF}
= \sigma_- \, (\rho \, P_N \, b^* \psi, P_N \, b^* \psi)_{\sF}
   + \sigma_+ \, (\rho \, P_N \, b \psi, P_N \, b \psi)_{\sF}$.
So $\lim_{N \to \infty} ((\checkQ_N \rho) \psi, \psi)_{\sF}
= ((\cQ \rho) \psi, \psi)_{\sF}$ for all $\rho \in \Psi(\gotC_1^{\rm sa})$.
If $N \in \Ni_0$, then
\begin{equation}
|(\cQ(P_N \, \rho \, P_N) \psi, \psi)_{\sF}|
\leq \sigma_- \, \|\rho\|_{\gotC_1} \, \|b^* \psi\|_{\sF}^2
   + \sigma_+ \, \|\rho\|_{\gotC_1} \, \|b \psi\|_{\sF}^2
\label{Skg3;50}
\end{equation}
for all $\rho \in \Psi(\gotC_1^{\rm sa})$, hence by continuity and density
of $\Psi(\gotC_1^{\rm sa})$ in $D(H_\sigma)$, the inequality
(\ref{Skg3;50}) is valid for all $\rho \in D(H_\sigma)$.
Therefore
$\lim_{N \to \infty} ((\checkQ_N \rho) \psi, \psi)_{\sF}
= ((\cQ \rho) \psi, \psi)_{\sF}$ for all $\rho \in D(H_\sigma)$.
So $(\checkQ_N)_{N \in \Ni_0}$
is a functional regularisation of $\cQ$.
It follows from the uniqueness in Corollary~\ref{ckg206} that
$(T^\sigma_t)_{t > 0}$ is the associated semigroup again.

\section*{Acknowledgements }

VAZ is grateful to Department of Mathematics of the University of Auckland and to
Tom ter Elst for a warm hospitality.
His visits were supported by the
EU Marie Curie IRSES program, project `AOS',
No.~318910 and by the Marsden Fund Council from
Government funding, administered by the Royal Society of New Zealand.

VAZ is also thankful to Alessandro Giuliani for a fruitful discussion on the
boson open systems,
which motivated him to consider a revision of the standard Kato regularisation.

\small

\noindent
{\sc A.F.M. ter Elst,
Department of Mathematics,
University of Auckland,
Private bag 92019,
Auckland 1142,
New Zealand}  \\
{\em E-mail address}\/: {\bf terelst@math.auckland.ac.nz}

\mbox{}

\noindent
{\sc Valentin Zagrebnov,
Institut de Math\'{e}matiques de Marseille - UMR 7373,
CMI-AMU, Technop\^{o}le Ch\^{a}teau-Gombert,
39, rue F. Joliot Curie, 13453 Marseille Cedex 13, France  \\
{\rm and} \\
D\'{e}partement de Math\'{e}matiques,
Universit\'{e} d'Aix-Marseille - Luminy, Case 901,
163 av.de Luminy, 13288 Marseille Cedex 09, France
}  \\
{\em E-mail address}\/: {\bf Valentin.Zagrebnov@univ-amu.fr}

\end{document}